\documentclass[12pt]{amsart}
\usepackage{amsmath,amssymb,latexsym}
\usepackage{hyperref}

\usepackage{color}

\definecolor{dark_purple}{rgb}{0.4, 0.0, 0.4}
\definecolor{dark_green}{rgb}{0.0, 0.7, 0.0}

\title[Factorizations induced by Nevanlinna-Pick factors]
{Factorizations induced by complete Nevanlinna-Pick factors}
\author[A. Aleman]{Alexandru Aleman}
\address{Lund University, Mathematics, Faculty of Science, P.O. Box 118, S-221 00 Lund, Sweden}
\email{alexandru.aleman@math.lu.se}

\author[M. Hartz]{Michael Hartz}
\address{Department of Mathematics, Washington University in St. Louis, One Brookings Drive,
	St. Louis, MO 63130, USA}
\email{mphartz@wustl.edu}
\thanks{M.H. was partially supported by an Ontario Trillium Scholarship and a Feodor Lynen Fellowship}

\author[J. M\raise.5ex\hbox{c}Carthy]{John E. M\raise.5ex\hbox{c}Carthy}
\address{Department of Mathematics, Washington University in St. Louis, One Brookings Drive,
	St. Louis, MO 63130, USA}
\email{mccarthy@wustl.edu}
\thanks{J.M. was partially supported by National Science Foundation Grant DMS 1565243}

\author[S. Richter]{Stefan Richter}
\address{Department of Mathematics, University of Tennessee, 1403 Circle Drive, Knoxville, TN 37996-1320, USA}
\email{richter@math.utk.edu}


\date{}

\def\H{{\mathcal H}}
\def\N{{\mathbb N}}
\def\C{{\mathbb C}}
\def\HC{{\mathcal C}}

\def\HB{{\mathcal B}}
\def\HD{{\mathcal D}}
\def\HM{{\mathcal M}}

\def\R{{\mathbb R}}

\def\D{{\mathbb D}}
\def\B{{\mathbb B}}

\def\T{{\mathbb T}}
\def\HL{{\mathcal L}}

\def\HE{{\mathcal E}}

\DeclareMathOperator{\Mult}{Mult}
\DeclareMathOperator{\Hol}{Hol}
\renewcommand{\Re}{\operatorname{Re}}

\newcommand{\Ker}[1]{\mathsf{Ker}~}

\newtheorem{theorem}{Theorem}[section]
\newtheorem{prop}[theorem]{Proposition}

\newtheorem{lemma}[theorem]{Lemma}
\newtheorem{corollary}[theorem]{Corollary}

\numberwithin{equation}{section}
\def\be{ \begin{equation}}
\def\ee{ \end{equation}}
\newenvironment{Pf}{\noindent{\textit{Proof of}}}{$\square$ }
\begin{document}
\begin{abstract} \noindent   We prove a factorization theorem for reproducing kernel Hilbert spaces whose kernel has a normalized complete Nevanlinna-Pick factor. This result relates the functions in the original space to pointwise multipliers  determined by the  Nevanlinna-Pick kernel   and has a number of interesting applications. For example, for a large class of spaces including Dirichlet and Drury-Arveson spaces, we construct for every function $f$ in the space a pluriharmonic majorant of $|f|^2$ with the property that whenever the majorant is bounded, the corresponding function $f$ is a pointwise multiplier.
 	\end{abstract}
 		\maketitle
	\section{Introduction}	
	
	Let $\Omega$ be a nonempty set and let $\HE$ be a separable  Hilbert space.  A function $\kappa:\Omega \times \Omega \to \HB(\HE)$ is called positive definite, if whenever $n\in \N$ and $w_1, ...w_n \in \Omega$ and $x_1,..., x_n\in \HE$, then $\sum_{i,j=1}^n \langle \kappa(w_j,w_i)x_i,x_j\rangle \ge 0$. If $\kappa$ is positive definite, then we write $\kappa_w(z)=\kappa(z,w)>>0$. We note that such $\kappa$ is positive definite, if and only if there is an auxiliary Hilbert space $\HC$ and a function $K:\Omega \to \HB(\HC,\HE)$ such that $\kappa_w(z)=K(z)K(w)^*$ for all $z, w\in \Omega$ (see \cite{AgMcC}, Theorem 2.62).
	
	If $k_w(z)$ is a scalar-valued reproducing kernel and if $I_\HE$ denotes the identity operator on $\HE$, then $k_w(z)I_\HE$ is a positive definite operator-valued kernel. It is the reproducing kernel for the space $\H_k(\HE)$ which consists of all functions $F:\Omega \to \HE$ such that for each $x\in \HE$ the function $F_x(z)=\langle F(z),x\rangle_\HE$ is in $\H_k$ and such that $\|F\|^2_{\H_k(\HE)}=\sum_{n}\|F_{e_n}\|^2_{\H_k}<\infty$, where $\{e_n\}$ is an orthonormal basis for $\HE$. It is easy to show that the expression for $\|F\|_{\H_k(\HE)}$ is independent of the choice of orthonormal basis, and that for each $F\in \H_k(\HE)$, $x\in \HE$, and $z\in \Omega$ one has $k_zx \in \H_k(\HE)$ and $\langle F(z),x\rangle_{\HE}=\langle F,k_zx\rangle_{\H_k(\HE)}$. It follows that the set of finite linear combinations of functions of the form $k_z x$, $z\in \Omega$, $x \in \HE$, is dense in $\H_k(\HE)$.
	Of course, the map: $f\otimes x \mapsto fx$ extends to be a Hilbert space isomorphism between $\H_k \otimes \HE$ and $\H_k(\HE)$, but for this paper we prefer the standpoint of $\HE$-valued functions.

If $k_w(z)$ and $s_w(z)$ are reproducing kernels on $\Omega$ and if $\HC$ and $\HD$ are separable Hilbert spaces, then $\Mult(\H_s(\HD),\H_k(\HC))$
is the collection of functions $\Phi:\Omega\to \HB(\HD,\HC)$ such that $(M_\Phi F)(z)=\Phi(z)F(z)$ defines a bounded operator $M_\Phi :\H_s(\HD)\to \H_k(\HC)$. It is easy to check that for $\Phi \in \Mult(\H_s(\HD),\H_k(\HC))$ one has
\be
\label{multiplieraction}
M_\Phi^*(k_w x)=s_w\Phi(w)^*x
\ee
  for all $x \in \HC$ and $w\in \Omega$. Moreover, the multipliers $\Phi$ with $\|M_\Phi\|\le A$ are characterized by
\be
\label{general-multiplier-norm}A^2 I_\HC k_w(z)-\Phi(z)\Phi^*(w)s_w(z)>>0,
\ee
since it is equivalent to $\|M_\Phi^*f\|_{\H_s(\HD)}\le A\|f\|_{\H_k(\HC)}$, for a dense subset of $\H_k(\HC)$.

As usual,  $\Mult(\H,\H)$ is denoted by  $\Mult(\H)$.
Each $\varphi\in \Mult(\H_k)$ defines a bounded operator on $\H_k(\HE)$, $(M_\varphi F)(z) =\varphi(z) F(z)$ and
$\|\varphi\|_{\Mult(\H)}=\|M_\varphi\|_{\HB(\H_k(\HE))}$, the  norm of $M_\varphi$ acting on  $\H_k$.
We will refer to such $M_\varphi\in \HB(\H_k(\HE))$ as scalar multiplication operators.

\
The central role in this work is played by scalar  reproducing kernels
$s:\Omega\times \Omega\to\C$, of the form
\be
\label{normalized-CNP}
s_w(z)= \frac{1}{1-\sum_{n=1}^\infty u_n(z)\overline{u_n(w)}}=\frac1{1-u(z)u^*(w)},\ee
 where $u_n:\Omega\to \C$, satisfy $u_n(z_0)=0$ for each $n\in \N$, and some fixed point $z_0\in\Omega$. Moreover,  $u:\Omega\to \HB(l^2,\C)$ denotes the corresponding row operator-valued function. By  a theorem of  Agler and McCarthy \cite{AM00} it follows that these are precisely the normalized ($s_{z_0}=1$) reproducing kernels with the \emph{complete Nevanlinna-Pick property}, such that the  Hilbert space $\H_s$ is separable. This class of kernels is well-known and extensively studied. A comprehensive treatment can be found in \cite{AgMcC}. The standard definition  based on finite interpolation problems is deferred to the  preliminary Section 2.
 The normalization point $z_0$ will be fixed for the rest of the paper, and throughout we shall refer to such kernels as normalized CNP kernels.  Clearly, the positivity of $s$ implies that   $u(z)$ is a strict  contraction when $z\in \Omega$. The first examples that come to mind are when $\Omega=\mathbb{B}_d,$ the unit ball in $\C^d$,  $d\in \mathbb{N}$, $z_0=0$, and $s_w(z)=\frac1{1-\langle z,w\rangle}$. If $d=1$, $\H_s$ is the standard Hardy space $H^2$, while for $d>1$, $\H_s$ is called the Drury-Arveson space and is denoted by $H_d^2$, but the list of examples is much larger.  It includes weighted Dirichlet spaces in one variable, weighted Besov spaces in one or several variables (on $\B_d$), and even Sobolev spaces (see   \cite{AgMcC} for more examples).

  \
 The present paper is concerned with  reproducing kernel Hilbert spaces $\H_k(\HE)$, with the property  that the kernel $k$ has a normalized CNP factor, i.e.
  \be
  \label{CNPfactor}
  k=sg,
  \ee
where $s$ is a normalized CNP kernel and  $g$ is positive definite.
It turns out that this condition is fulfilled for a large class of  kernels,  for example Hardy and weighted Bergman spaces on $\B_d$, or the polydisc $\B_1^d$,  of course $\H_s$ itself, or more generally,  $\H_{s^t},~t\ge 1$. In Section 2 we shall show that whenever  $\Mult(\H_k)$ contains non-constant elements  the kernel $k$ has a nontrivial normalized CNP factor.   Our purpose is to establish a factorization result for functions in $\H_k(\HE)$ which is motivated by a simple observation. If we assume that
$$k_w(z)=s_w(z)G(z)G^*(w), \quad s_w=\frac1{ 1-u(z)u^*(w)},$$ with $z,w\in \Omega,~ G(z)\in \HB(\HC,\C),~u(z)\in \HB(l^2,\C)$,
it is not difficult to verify  $G\in \Mult(\H_s(\HC),\H_k)$, $u\in \Mult(\H_s(l^2),\H_s)$   and that both have  multiplier norm at most $1$.
 This leads to the factorization
 $$k_w=\frac{\varphi}{1-\psi},$$
 where  $\varphi(z)= G(z)G(w)^*\in \Mult(\H_s,\H_k),~\psi(z)=u(z)u(w)^*\in \Mult(\H_s)$ are  multipliers with $\psi(z_0)=0$, $\|\psi\|_{\Mult(\H_s)}\le 1$.\\

Our main theorem shows that this factorization continues to hold for arbitrary elements of $\H_k(\HE)$, and establishes a  sharp estimate  for the multiplier norms involved. Somewhat surprisingly,   our factorization is unique for elements of
unit norm.

  \begin{theorem}\label{factorization theorem}  Let $\Omega$ be a non-empty set, $z_0\in \Omega$. Let $k_w(z)$ be a reproducing kernel on $\Omega$ that can be factored in the form $k_w(z)= s_w(z) g_w(z)$ where $s_w(z)$ is a normalized CNP kernel with $s_{z_0}=1$  and $g_w(z)>>0$.\\
  (i) For  $F:\Omega\to \HE$, the following are equivalent:
  	
  	(a) $F\in \H_k(\HE)$  with $\|F\|_{\H_k(\HE)}\le 1$,
  	
  	(b)  There is a $\psi \in \Mult(\H_s)$ with $\psi(z_0)=0$ and  a $\Phi\in \Mult(\H_s,\H_k(\HE))$ such that
  	$$\|\psi h\|^2_{\H_s}+ \|\Phi h\|^2_{\H_k(\HE)} \le \|h\|^2_{\H_s} \ \text{ for all }h \in \H_s,$$ and  $F(z) = \frac{1}{1-\psi(z)}\Phi(z)$ for all $z\in \Omega$.\\
 (ii) 	If $F\in \H_k(\HE)$  with $\|F\|_{\H_k(\HE)}= 1$, then the factorization given in (b) is unique. In fact, $s_z\in \Mult(\H_k(\HE))$, and  if  $V_F(z) =2\langle F, s_zF\rangle_{\H_k(\HE)}-1,~z\in \Omega$, then $\text{\rm Re }V_F\ge 0$ in $\Omega$  and  (b) holds with
$$\psi=\frac{V_F-1}{V_F+1},\quad \Phi=\frac{2}{V_F+1} F.$$  	
  \end{theorem}
From   \cite[Lemma 1]{AHMcR} it follows that if $\psi$ is the multiplier given in part (i) of the theorem then $1 - \psi$  is cyclic, i.e. the multiplier-invariant subspace generated by this function equals $\H_s$. It will also turn out from the proof that if $\|F\|_{\H_k(\HE)}<1$, the factorization is far from unique. In fact if   (b) holds, the restriction imposed on  $\|F\|_{\H_k(\HE)}$ is related only to the condition $\psi(z_0)=0$. Our argument shows (see Proposition \ref{a-factorization} below) that  without this condition the factorization $F=\frac{1}{1-\psi}\Phi$ with $\Phi,\psi$ as above,  holds true for an arbitrary $F\in \H_k(\HE)$.

\
In the case when $k=s$ and $\HE=\C$, the factorization in (b)   was recently proved in \cite{AHMcR}. The result was previously  shown for the Drury-Arveson space on $\mathbb{B}_d$,
  by Alpay, Bolotnikov and Kaptano\u{g}lu
  \cite[Theorem 10.3]{ABK02}  while for  the Dirichlet space, the corresponding  question was  posed in \cite[Section 3]{Ross06} and at the end of \cite{MR15}.
  The proof of the result in \cite{AHMcR}  relies on
  an appropriate  version of Leech's theorem, and at the beginning of Section 3 we shall briefly indicate how one can use that method to extend the result to the setting considered here. However, our proof of Theorem \ref{factorization theorem} follows a different path, namely the one suggested by part (ii). The argument is   based on an idea of Sarason (see \cite{Sarason1,Sarason2}) and further developments in \cite{GrRiSu} and will be presented in Section 3. It leads to a constructive approach which has a number of interesting applications  given in  Section 4. The function $V_F$ defined in part (ii) of the main theorem plays a crucial role for our development, and it will be called \emph{the Sarason function of $F$}. Its real part behaves similarly to the well known \emph{least harmonic majorant}
from the theory of Hardy spaces. We remark that in the case when $k=s$ and $\HE = \mathbb C$, this argument
also provides a more explicit construction of the factorization in the main result of \cite{AHMcR}. \\
In Section 4  we show that multiplier-invariant subspaces in $\H_k$ are generated by functions in $\Mult(\H_s,\H_k)$,   that  extremal functions in $\H_k(\HE)$ belong to $\Mult(\H_s,\H_k(\HE))$, and derive a pointwise estimate for these functions.  These  recover some results in \cite{McT} and in the recent paper   \cite{Eschetal}, but  apply to other situations as well.   The Sarason function of an extremal element is constant
equal to $1$ in $\Omega$, and motivated by this observation we continue the investigation of  $\Mult(\H_s,\H_k(\HE))$ in terms of this object. It turns out that if the one-function corona theorem holds in $\Mult(\H_s)$, then  $F\in
\Mult(\H_s,\H_k(\HE))$ whenever $V_F$ is bounded in $\Omega$. \\One of our main applications shows that for  a large class of  kernels $k$  the same conclusion holds under the weaker assumption that the real part of the Sarason function is bounded.  The corresponding class of spaces $\H_k$ contains
Bergman, Hardy  and weighted Besov spaces on the ball or polydisc, in particular the Drury-Arveson spaces $H_d^2,~d\in \N$. In all of these cases, our Theorem \ref{diff-op-thm} together with Corollary  \ref{gen Harmonic Majorant} show a surprising analogy  to the classical $H^2$-case: \\The real part of the Sarason function of $F$ is a  majorant  of $\frac{s_z(z)}{k_z(z)}\|F(z)\|^2_\HE$ (in most cases pluriharmonic), such that whenever this majorant is bounded, $F$ belongs to $\Mult(\H_s,\H_k)$. The converse of this statement fails to be true. Based on the work of Shimorin (\cite{Shimorin},\cite{Shimorin2}), we construct multipliers of the standard weighted Dirichlet spaces on the unit disc whose Sarason functions have unbounded real part.
\\
Our theorem about multipliers applies also to Carleson embeddings. More precisely,
in the special case when $\H_k$ is a weighted Bergman space and $F=1$ this extends to the general context a sufficient condition for such embeddings  obtained recently in \cite{MaRa} for the Dirichlet space. 
 As pointed out above, in Section 2 we gather  some useful  preliminary results.

\section{Preliminaries}	
 \subsection{The complete Nevanlinna-Pick property}
 Given a reproducing kernel $s$ on the non-void set $\Omega$, we say that $s$ is a \emph{complete Nevanlinna- Pick kernel} if $\H_s$ has the following property: For every $r \in \mathbb{N}$ and every finite collection of points $z_1,\ldots,z_n \in X$
 and matrices $W_1,\ldots,W_n \in M_r(\C)$,
 positivity of the $nr \times nr$-block matrix
 \begin{equation*}
 \Big[ s(z_i,z_j) (I_{\C^r} - W_i W_j^*) \Big]_{i,j=1}^n
 \end{equation*}
 implies that there exists $u \in \Mult(\H_s(\C^r))$ of norm at most $1$ such that
 \begin{equation*}
 u(z_i) = W_i \quad (i=1,\ldots,n).
 \end{equation*}
 Such kernels were characterized
 by a theorem of Quiggin \cite{Quiggin93} and McCullough \cite{McCullough92}. Complete Nevanlinna-Pick kernels $s$  can be normalized at any point provided that  $s_w(z) \neq 0$ for all $z,w \in \Omega$ (see \cite[Section 2.6]{AgMcC}). 
 However, in this paper we shall only use the form \eqref{normalized-CNP} which follows from the Agler-McCarthy theorem in \cite{AM00}.  For further purposes we record  an elementary  result whose proof is included for the sake of  completeness.

\begin{lemma}\label{CNP} If $s$ is a normalized CNP kernel with $$s_w(z)= \frac{1}{1-\sum_{n=1}^\infty u_n(z)\overline{u_n(w)}},$$ where $u_n:\Omega\to \C$ and $u_n(z_0)=0$,
  then:\\
(i)  $u_n \in \Mult(\H_s)$ for each $n\in \mathbb{N}$ and if  $h_n\in\ \H_s$, then $$\|\sum_n u_n h_n\|^2\le \sum_n \|h_n\|^2.$$
(ii) $I-\sum_iM_{u_i}M_{u_i}^* = P_0$, where $P_0$ is the projection onto $s_{z_0}=1$, i.e.  $h - \sum_iM_{u_i}M_{u_i}^*h=h(z_0)$ for all $h\in \H_s$.	
\end{lemma}

\begin{proof} (i)  From the identity $$(1-\sum_n u_n\overline{u_n(w)}) s_w=s_{z_0}=1>>0,$$
we see that
\eqref{general-multiplier-norm}
 holds with $A=1$ and $\Phi:\Omega\to \HB(l^2,\C),~\Phi(z)(x_n)=\sum_nu_n(z)x_n$.
(ii) holds for reproducing kernels, hence it holds for all functions in $\H_s$ because the span of reproducing kernels is dense in $\H_s$ and the operator on the right hand side is bounded. The second part of (ii) is just a reformulation of the first.
\end{proof}

\subsection{Kernels with a CNP factor}

The simplest examples of kernels $k$  with the normalized CNP factor $s$ are given by $k=s^t,~ t\ge 1$.
Note that if $s=\frac1{1-uu^*}$ and $0 < t < 1$, then
$$s^{t}=\sum a_k(t)(uu^*)^k,$$ with $a_k(t)>0$, hence $s^{t}, ~0<t<1$  is positive definite. Then  $s^t$ is positive definite for all $t>0$ by the Schur product theorem.

The most general condition for a factorization of the form \eqref{CNPfactor} is as follows.
\begin{lemma} \label{general-criterion} Let $s$  be a normalized CNP kernel with $s_w(z)= \frac{1}{1- u(z)u^*(w)}$, with $u(z)\in\HB(l^2,\C)$.
The kernel $k$ can be factored as $k=sg$ with $g$ positive definite  if and only if   $u\in \Mult(\H_k(l^2),\H_k)$ with $\|M_u\|\le 1$.
 In this case $\Mult(\H_s(\HE))$ is contractively  contained in $\Mult(\H_k(\HE))$.
\end{lemma}
\begin{proof}
 To see the first part,  use  the identity
\[
(I - u(z) u^*(w) ) k_w(z) = \frac{1}{s_w(z)} k_w(z) ,\]
to conclude that $k/s$ is positive definite if and only if the left hand side is and then  apply \eqref{general-multiplier-norm}.
For the second part,  note that  if $\varphi\in \Mult(\H_s(\HE))$ with $\|M_\varphi\|=1$ then $(I_\HE-\varphi(z)\varphi^*(w))s_w(z)>>0$, hence by the Schur product theorem $$ (I_\HE-\varphi(z)\varphi^*(w))s_w(z)g_w(z)=(I_\HE-\varphi(z)\varphi(w)^*)k_w(z)>>0.$$
\end{proof}
It is interesting to note that the second part of the lemma can be used to characterize complete Nevanlinna-Pick kernels (see \cite[Theorem 4.4]{ClHa}).
With the result  in hand we can list some  further examples of kernels with a normalized CNP factor. Recall that a $d$-contraction on the Hilbert space $\H$ is a commuting tuple $(T_1,\ldots,T_d), ~T_j\in \HB(\H)$, with $\sum_{j=1}^dT_jT_j^*\le I$.

\begin{corollary}
	\label{exCNPfactor}
	(i)
  If $\Mult(\H_k(l^2), \H_k)$  contains a nonzero element $u= (u_n)_{n\ge 1}$ of norm at most $1$
   with $u(z_0)=0$ and   $s_w(z)= \frac{1}{1-\sum_{n=1}^\infty u_n(z)\overline{u_n(w)}}$, then $k/s>>0$.\\
	(ii) Let $\Omega =\mathbb{B}_d,~d\in \N$ and assume that multiplication by the coordinates forms a $d$-contraction on $\H_k$. If $s=\frac1{1-\langle z,w\rangle}$  then $k/s>>0.$\\
	(iii) If the assumption in (ii) holds for $d=1$, i.e. multiplication by the identity function is
  a contraction on $\H_k$, then for every normalized CNP kernel $s$ such that $\mathcal H_s$ consists
  of analytic functions, we have $k/s>>0$.
\end{corollary}

\begin{proof} (i) and (ii) follow directly from Lemma \ref{general-criterion}.
(iii) Let $s$ be a normalized CNP kernel with $s_w(z)= \frac{1}{1-\sum_{n=1}^\infty u_n(z)\overline{u_n(w)}}$, where
each $u_n$ is analytic. Write $k/s=(k/s_0)(s_0/s)$, where $s_0$ is the Szeg\"o kernel. By the Schur product theorem it suffices to show that $s_0/s>>0$, since $s/s_0 >> 0$ follows from the assumption. This is obviously equivalent to the operator inequality  $$\sum_{n=1}^\infty M_{u_n}M^*_{u_n}\le I $$
in $H^2$. But in this space we have 	
$$\sum_{n=1}^\infty M_{u_n}M^*_{u_n}\le \sum_{n=1}^\infty M_{u_n}^*M_{u_n}$$
and the right hand side is just the Toeplitz operator with symbol $\sum_{n=1}^\infty |u_n|^2\le 1$, i.e.  	
$$\sum_{n=1}^\infty M_{u_n}^*M_{u_n}\le I.$$	
	\end{proof}

Other examples are provided by the following result.
\begin{prop}\label{invariant-subspaces}  Let $s$ be a normalized CNP kernel on $\Omega$ and $k$ be a reproducing kernel with $k/s>>0.$ If $\HM$ is a closed subspace of $\H_k$ which is invariant for $\Mult(\H_s)$, then the reproducing kernel $k^\HM$ of $\HM$ satisfies $k^\HM/s>>0$.
	\end{prop}
	\begin{proof}
    If  $s_w(z)= \frac{1}{1-\sum_{n=1}^\infty u_n(z)\overline{u_n(w)}}$ then by Lemma \ref{CNP} and Lemma \ref{general-criterion} we have that $$	 \sum_{n=1}^\infty \|u_nh_n\|^2_{\mathcal H_k} \le 	\sum_{n=1}^\infty \|h_n\|^2_{\mathcal H_k}$$
		if  $h_n\in \H_k$, in particular if $h_n\in \HM$. Thus  $M_u:\HM(l^2)\to \HM$ is a contraction and the result follows by another application of Lemma \ref{general-criterion}.
			\end{proof}

\section{Proof of the main result}
Before we give the actual proof we shall discuss briefly some related ideas as well as  the motivation for our approach.
\subsection{(a) $\Rightarrow$ (b) via Leech's  theorem}
As pointed out in the Introduction, this part of  Theorem \ref{factorization theorem}, that is, the representation $F=\frac1{1-\psi}\Phi$, where $\|F\|_{\H_s(\HE)}\le 1$,  $\psi \in \Mult(\H_s)$ with $\psi(z_0)=0$ and   $\Phi\in \Mult(\H_s,\H_k(\HE))$ with
$$\|\psi h\|^2_{\H_s}+ \|\Phi h\|^2_{\H_k(\HE)} \le \|h\|^2_{\H_s} \ \text{ for all }h \in \H_s,$$
can be proved with the method in \cite{AHMcR} and we shall describe briefly this approach.\\
 In \cite{AHMcR} the factorization theorem is proved in the scalar case and when $k=s$ with help of  an appropriate version of  Leech's theorem, which is the implication
(i) $\Rightarrow$ (ii) of \cite[Theorem 8.57]{AgMcC}. This can easily be adapted for normalized CNP kernels. A careful inspection of the argument shows that it extends to the vector-valued case as well. Moreover, even if the approach is not constructive, the method gives  the multiplier-norm estimates when $k=s$.

\
The general case of a kernel $k$ with $k/s>>0$, where  $s$ is  a normalized CNP kernel, can be deduced as follows. Let \be
\label{k=sGG-star}k_w(z)=s_w(z)G(z)G^*(w),\ee
with $G:\Omega\to\HB(\HC,\C)$ and define
$\tilde{G}:\Omega\to \HB(\HC\otimes \HE,\HE)$ by $$\tilde{G}(z)(x\otimes y)=(G(z)x)y.$$
Observe first that $\tilde{G}$ is a contractive multiplier from $\H_s(\HC\otimes \HE)$ into $\H_k(\HE)$.
Indeed,   $\tilde{G}(z)^* y=G(z)^*\otimes y$,  hence from \eqref{k=sGG-star} it follows that the map  defined on a spanning set by $M_{\tilde{G}}^*(k_zy)=s_z (G(z)^*\otimes y)$ extends to an isometric  operator $M_{\tilde{G}}^*: \H_k(\HE) \to\H_s(\HC\otimes \HE)$. Thus $M_{\tilde{G}}^{**}= M_{\tilde{G}}:\H_s(\HC\otimes \HE) \to \H_k(\HE)$ is a contraction which proves the claim.\\
Now let  $F\in \H_k(\HE)$ with $\|F\|_{ \H_k(\HE)}\le 1$, and let $H=M_{\tilde{G}}^*F$. Then $H\in \H_s(\HC\otimes \HE)$,
 $\|H\|_{\H_s(\HC\otimes \HE)}\le 1,$ and $F=M_{\tilde{G}}H$, because $M_{\tilde{G}}^*$ is an isometry. By the previous discussion,
 there are  $\psi \in \Mult(\H_s)$ and $\Gamma\in \Mult(\H_s,\H_s(\HC\otimes \HE))$ such that $\psi(z_0)=0$, $\|\psi h\|^2+\|\Gamma h\|^2\le \|h\|^2$ for every $h\in \H_s$ and $H(z)= \frac{1}{1-\psi(z)}\Gamma(z)$, which gives $F(z)=\frac{1}{1-\psi(z)}\Phi(z)$
with
 \be
 \label{phi=GGam}\Phi(z)=\tilde{G}(z)\Gamma(z),\ee
and the result follows from the fact that $\tilde{G}$ is a contractive multiplier from $\H_s(\HC\otimes \HE)$ into $\H_k(\HE)$.

\
As the  following result shows, this reasoning leads to additional information about the factorization in Theorem \ref{factorization theorem}. For simplicity, we shall consider only the scalar-valued case.
\begin{corollary}\label{phi=GGam-appl} Let $s$ be a normalized CNP kernel on $\Omega$ and let  $k=sg$, with $$g_w(z)=\sum_{n=0}^\infty g_n(z)\overline{g}_n(w).$$
 A function $f$ belongs to the unit ball of $\H_k$ if and only if there exist $\psi,\varphi_n \in \Mult(\H_s),~n\ge 1$,  with $\psi(z_0)=0$ and  $$\|\psi h\|_{\H_s}^2+\sum_n\|\varphi_n h\|_{\H_s}^2\le \|h\|_{\H_s}^2,\quad h\in \H_s,$$
such that $$f=\frac{\sum_ng_n\varphi_n}{1-\psi}.$$ \end{corollary}
\begin{proof} This is a direct application of
 Theorem  \ref{factorization theorem} (i) together with the equality \eqref{phi=GGam}.  Here $\HE=\C$ and $\Gamma$ becomes an element of $\Mult(\H_s,\H_s(l^2))$.\end{proof}
Some concrete examples of this type are discussed in  Section 4.

\subsection{The constructive approach}		
As already mentioned, our approach is different and it is based on an idea of Sarason \cite{Sarason1,Sarason2} which was further developed in \cite{GrRiSu}. Here is a short  motivation for it. Recall from Lemma \ref{general-criterion} that for each $z\in \Omega$, $s_z\in \Mult(\H_k(\HE))$. Therefore,  for  $F\in \H_k(\HE)$, the Sarason function of $F$   \be
\label{V-F} V_F(z) =2\langle F, s_zF\rangle_{\H_k(\HE)}-\|F\|_{\H_k(\HE)}^2,~z\in \Omega,\ee
is well defined.
A straightforward computation shows that if $s$ is the Szeg\"o kernel and $k=s$ then $\text{Re }V_F$ is just the Poisson integral of
$\|F\|_{\HE}^2$, hence  it satisfies $$0\le \|F(z)\|_{\HE}^2\le \text{Re }V_F(z).$$
It is the least harmonic majorant of $\|F\|_\HE$.
 A remarkable fact proved in \cite[Section 2]{GrRiSu}   is that this inequality continues to hold for arbitrary normalized CNP kernels $s$, more precisely
\be
\label{Harmonic Majorant}0\le \|F(z)\|_{\HE}^2\le \frac{\|s_z F\|_{\H_s(\HE)}^2}{\|s_z\|^2}\le  \text{Re }V_F(z), \ee
in particular, $\text{Re }V_F$ is  a majorant of $\|F\|_\HE$ as well. Moreover,  if $\Omega\subset\C^d$,  $\H_s$ consists of analytic functions, $V_F$ is  harmonic when $d=1$, or  pluriharmonic when $d>1$.
Now assume that $\|F\|_{\H_s(\HE)}=1$, set  $\psi=\frac{V_F-1}{V_F+1},\quad \Phi=\frac{2}{V_F+1} F$. Then
$\psi(z_0)=0$, $F=\frac1{1-\psi}\Phi$ and
 a straightforward computation shows that  	
$$|\psi(z)|^2 +\|\Phi(z)\|_{\HE}^2= \frac{|V_F(z)|^2-2\text{ Re }V_F(z) + 1 + 4 \|F(z)\|_{\HE}^2}{|V_F+1|^2}\le 1.$$
For example, if $\HE=\C$ and $M(\H_s)=H^\infty$ with equality of norms, then this proves part of our main theorem, and in fact for a single function $f\in H^2$ this  proof was given by  Sarason in \cite{Sarason1,Sarason2}.
In the general case considered here, pointwise estimates as above cannot lead to a proof of the main theorem. However, the intuition behind our approach is the argument outlined here.
\subsection{The proof of Theorem \ref{factorization theorem}}
For the remainder of this section, let $k$ and $s$ be reproducing kernels on $\Omega$ as
in the statement of Theorem \ref{factorization theorem}, i.e.\ $s$ is a CNP kernel,
normalized at $z_0$, and $k = s g$ with $g >>0$.
The key step is the following far-reaching generalization of the inequality \eqref{Harmonic Majorant}.
\begin{lemma}\label{main-lemma-sarason}  Let $F \in \H_k(\HE)$. Then
	$$ \langle s_w F,F\rangle_{\H_k(\HE)} +\langle F, s_z F\rangle_{\H_k(\HE)} -\|F\|_{\H_k(\HE)}^2- \frac{\langle s_w F,s_z F\rangle_{\H_k(\HE)}}{s_w(z)} >> 0.$$
\end{lemma}
\begin{proof} Let $s_w(z)= \frac{1}{1-u_w(z)}$ with $u_w(z)=\sum_n \overline{u_n(w)}u_n(z)$, $u_n(z_0)=0$ for all $n$.
	Then $\frac{1}{s_w(z)}=1-u_w(z)$ and hence
	\begin{align}\label{crucial step-pos}  \langle s_w F,F\rangle_{\H_k(\HE)} &+\langle F, s_z F\rangle_{\H_k(\HE)} -\|F\|_{\H_k(\HE)}^2- \frac{\langle s_w F,s_z F\rangle_{\H_k(\HE)}}{s_w(z)}
	\\
	&\nonumber= u_w(z) \langle s_w F,s_z F\rangle_{\H_k(\HE)} - \langle (s_w-1)F,(s_z-1)F\rangle_{\H_k(\HE)}.
	\end{align}
	
	Let $z_1,...,z_k\in \Omega$, $a_1,...,a_k\in \C$,  set $h=\sum_{i=1}^k a_is_{z_i}$, and $h_n= \sum_i a_i \overline{u_n(z_i)}s_{z_i}$ for $n=1,2, ...$. Then $h,h_n\in  \Mult(\H_s)\subset  \Mult(\H_k(\HE))$,	
	\begin{align*} \sum_n\|h_nF\|_{\H_k(\HE)}^2&=\sum_n \|\sum_i a_i \overline{u_n(z_i)}s_{z_i}F\|_{\H_k(\HE)}^2\\ &=\sum_{i,j} a_i\overline{a}_j u_{z_i}(z_j) \langle s_{z_i}F,s_{z_j}F\rangle_{\H_k(\HE)}. \end{align*}
Moreover,  since  $h-h(z_0)=\sum_i a_i (s_{z_i}-1)$,	
$$\|(h-h(z_0))F\|_{\H_k(\HE)}^2 =\sum_{i,j} a_i\overline{a}_j \langle (s_{z_i}-1)F,(s_{z_j}-1)F\rangle_{\H_k(\HE)}.$$
 	By Lemma \ref{CNP} (ii) we have $h-h(z_0)=\sum_n M_{u_n} M_{u_n}^*h= \sum_n u_n  h_n$
	and hence by Lemma \ref{general-criterion}
	$$\|(h-h(z_0))F\|_{\H_k(\HE)}^2 = \|\sum_n u_n(h_nF)\|_{\H_k(\HE)}^2\le \sum_{n}\|h_n F\|_{\H_k(\HE)}^2.$$
	Thus, by \eqref{crucial step-pos}
	\begin{align*} \sum_{i,j}  a_i\overline{a}_j &\left(\langle s_{z_i} F,F\rangle_{\H_k(\HE)} +\langle F, s_{z_j} F\rangle_{\H_k(\HE)}
 -\|F\|_{\H_k(\HE)}^2- \frac{\langle s_{z_i} F,s_{z_j} F\rangle_{\H_k(\HE)}}{s_{z_i}({z_j})}\right)\\
	&= \sum_{i,j}  a_i\overline{a}_j \left( u_{z_i}(z_j) \langle s_{z_i} F,s_{z_j} F\rangle_{\H_k(\HE)} - \langle (s_{z_i}-1)F,(s_{z_j}-1)F\rangle_{\H_k(\HE)}\right)\\
	&= \sum_n\|h_nF\|_{\H_k(\HE)}^2 - \|(h-h(z_0))F\|_{\H_k(\HE)}^2 \ge 0.
	\end{align*}\end{proof}

An immediate application of the lemma yields  the  general version of the inequality \eqref{Harmonic Majorant}.

\begin{corollary}\label{gen Harmonic Majorant}  Let $s,k$ be reproducing kernels on the nonempty set $\Omega$ such that $s$ is a normalized CNP kernel and  $k=sg$ with $g>>0$. 	
Let $F \in \H_k(\HE)$  and let $V_F$ be the Sarason function of $F$  given by \eqref{V-F}.
Then
$$\frac{s_z(z)}{k_z(z)}\|F(z)\|_{\HE}^2\le \text{\rm Re }V_F(z).$$
 \end{corollary}
\begin{proof}  Apply Lemma \ref{main-lemma-sarason} to obtain
$$\frac{V_F(z)+\overline{V_F}(w)}{2}- \frac{\langle s_w F,s_z F\rangle_{\H_k(\HE)}}{s_w(z)}>>0.$$
For $z=w\in\Omega$ this gives
$$\frac{\| s_z F\|_{\H_k(\HE)}^2}{s_z(z)}\le \text{Re}V_F(z),$$
and the standard estimate $$
 s_z(z)^2\|F(z)\|_{\HE}^2\le k_z(z) \| s_z F\|_{\H_k}^2,$$
gives  the inequality in the statement.
	\end{proof}

Our next step relates the positivity result in Lemma \ref{main-lemma-sarason} to the multiplier-norm estimates in Theorem \ref{factorization theorem}

\begin{lemma} \label{|f|^2<=Re Phi}   Let  $f_1,\ldots,  f_N\in  \H_k$ be finite linear combinations of reproducing kernels in $\H_k$, let $\{e_n\}$ be an orthonormal basis in $\HE$ and set
$$F(z)=\sum_{n=1}^Nf_n(z)e_n.$$
Then $F\in \Mult(\H_s,\H_k(\HE))$  and its Sarason function  $V_F$  belongs to $\Mult(\H_s)$.  Moreover, if $h\in \H_s$  then
	 \be
	 \label{sarason-step1} \|hF\|_{\H_k(\HE)}^2\le \text{\rm  Re } \langle V_F  h,h\rangle_{\H_s},\ee
and if $a\in\C,~\text{\rm Re }a>0$,
\be
\label{sarason-step2} \|(V_F-a)h\|_{\H_s}^2+4 \Re a \, \|hF\|_{\H_k(\HE)}^2\le \|(V_F+\overline{a})h\|_{\H_s}^2.\ee
\end{lemma}

\begin{proof} When $F$ has the special form given in the  statement, its Sarason function can be written as
$$V_F(z)=2\sum_{i,j=1}^mc_{ij}k_{z_i}(z_j)s_{z_i}(z)-\|F\|_{\H_k(\HE)}^2$$
for a suitable choice of scalars $c_{i j}$,
which implies that $V_F\in \Mult(\H_s)$. If $h= \sum_i a_i s_{z_i}$, then  $h\in \Mult(\H_k)$  and \eqref{sarason-step1} is equivalent to
\begin{align*}\sum_{i,j}&a_i\overline{a}_j\langle s_{z_i}F,s_{z_j}F\rangle_{\H_k(\HE)} \le \frac1{2}\sum_{i,j}a_i\overline{a}_js_{z_i}(z_j)
(V_F(z_j)+\overline{V_F(z_i)})\\&
=\sum_{i,j}a_i\overline{a}_js_{z_i}(z_j)(\langle F, s_{z_j}F\rangle_{\H_k(\HE)}+ \langle s_{z_i} F,F\rangle_{\H_k(\HE)}-\|F\|_{\H_k(\HE)}^2)
\end{align*}
and the inequality follows by an application of Lemma \ref{main-lemma-sarason} and the Schur product theorem.  Moreover, \eqref{sarason-step2} is just a reformulation of \eqref{sarason-step1}, since  $$ \|(V_F+\overline{a})h\|_{\H_s}^2- \|(V_F-a)h\|_{\H_s}^2=
4 \Re a \Re \langle V_F  h,h\rangle_{\H_s}. $$ Finally, \eqref{sarason-step1} together with the fact that finite linear combinations of reproducing kernels in $\H_s$ are dense in $\H_s$,
shows that  $F\in \Mult(\H_s,\H_k(\HE))$, and that both inequalities hold for arbitrary $h\in \H_s$. \end{proof}

We can now prove the factorization   in Theorem \ref{factorization theorem} (ii) in a slightly more general form, which turns out to be useful in applications.

\begin{prop}\label{a-factorization} Let $F\in \H_k(\HE)$  and let $a\in \C$ with $\text{\rm Re }a>0$. If $V_F$ is the Sarason function of $F$, $$\psi_a=\frac{V_F-a}{V_F+\overline{a}},\quad \Phi_a=\frac{2}{V_F+\overline{a}}F,$$
then $F=\frac{\text{\rm Re }a}{1-\psi_a}\Phi_a$, $\psi_a\in \Mult(\H_s),~\Phi_a\in \Mult(\H_s, \H_k(\HE))$ with
$$\|\psi_ah\|_{\H_s}^2+\text{\rm Re }a \|\Phi_ah\|_{ \H_k(\HE)}^2  \le\|h\|_{\H_s}^2,\quad h\in \H_s.$$
\end{prop}

\begin{proof}
Assume first that $F$ has the form in Lemma \ref{|f|^2<=Re Phi}	and recall from the lemma that in this case  $V_F\in \Mult(\H_s)$. If $h=(V_F+\overline{a})u$ with $u\in \H_s$, the inequality  \eqref{sarason-step2} applied to $u$ can be rewritten as
$$\left\|\frac{V_F-a}{V_F+\overline{a}}h\right\|_{\H_s}^2+\text{\rm Re }a\left\|\frac{2}{V_F+\overline{a}}F h\right\|_{ \H_k(\HE)}^2   \le\|h\|_{\H_s}^2,$$
which is precisely the inequality in the statement.
We claim that $M_{V_F+\overline{a}}$ has dense range. Indeed, if $h\in (M_{V_F+\overline{a}}\H_s)^\perp$, then $$\langle (M_{V_F+\overline{a}}+M_{V_F+\overline{a}}^*)h,h\rangle_{\H_s}=0,$$
and by  \eqref{sarason-step1} it follows that $$0=\text{Re}\langle(V_F+\overline{a})h,h\rangle_{\H_s}\ge \|h\|_{\H_s}^2\text{Re }a+ \|Fh\|_{\H_k(\HE)}^2,$$
which implies that $h=0$ and proves the claim.
 Thus for $F$ as above we obtain that  $\psi_a\in \Mult(\H_s),~\Phi_{a}\in \Mult(\H_s,\H_k(\HE))$, and the inequality in the statement holds  for all $h\in\H_s$.\\
Clearly, every $F\in \H_k(\HE)$ can be approximated in  $\H_k(\HE)$ by a sequence $(F_N)$ as above and from the previous argument we have that $\psi_a^N=
\frac{V_{F_N}-a}{V_{F_N}+\overline{a}}\in \Mult(\H_s),~ (\sqrt{\text{Re }a})\Phi_a^N=\frac{2 \sqrt{\text{Re }a}}{V_{F_N}+\overline{a}}F_N\in \Mult(\H_s,\H_k(\HE))$
are contractive multipliers. Also   note that $(V_{F_N})$ converges pointwise to $V_F$ in $\Omega$.\\
 Then for $h\in \H_s$, $(\psi_a^Nh)$ converges weakly in $\H_s$ to $\psi_ah$. Similarly,  $(\Phi_a^Nh)$ converges weakly in $\H_k(\HE)$ to $\Phi_ah$, because the sequence is bounded and satisfies
$$\lim_{N\to\infty}\langle \Phi_a^Nh, k_ze\rangle_{\H_k(\HE)}=\lim_{N\to\infty}\frac{2}{V_{F_N}(z)+\overline{a}}\langle F_N(z),e\rangle_{\HE}=\langle \Phi_ah, k_ze\rangle_{\H_k(\HE)},$$
for all $z\in \Omega,~e\in \HE$.
Thus  $(\psi_a^N h, \Phi_a^Nh)$ converges weakly in $\H_s\oplus \H_k(\HE)$ which implies
$$ \|\psi_ah\|_{\H_s}^2+ \text{\rm Re }a\|\Phi_ah\|_{ \H_k(\HE)}^2  \le
\liminf_{N\to\infty}(\|\psi^N_ah\|_{\H_s}^2+  \text{\rm Re }a\|\Phi^N_ah\|_{ \H_k(\HE)}^2) \le\|h\|_{\H_s}^2,$$
and completes the proof.
\end{proof}
Note that the factorization holds without any assumption on $\|F\|_{\H_k(\HE)}$, but we do not control the value  $\psi_a(z_0)$.

\
The factorization in  Theorem \ref{factorization theorem}  (ii) is a direct application. If $\|F\|_{\H_k(\HE)}=1$,  the result is obtained for $a=1$, $\psi=\psi_1,~\Phi=\Phi_1$.
Note that since  $V_F(z_0)=1$,  we have $\psi(z_0)=0$.  For $\|F\|_{\H_k(\HE)}<1$, the factorization, and hence the implication (a) $\Rightarrow$ (b) in Theorem \ref{factorization theorem}  (i),  is obtained as follows. \\
Let
 $w\in \Omega$  and  apply the previous argument  to the function  $$F_w=\left(F,\sqrt{\frac{1-\|F\|_{\H_k(\HE)}^2}{k_w(w)}} k_w\right)\in \H_k(\HE\oplus\C),$$
which has  unit norm in the space above. The Sarason  function of  $F_W$ is  $$V_{F_w}=V_F+(1-\|F\|_{\H_k(\HE)}^2)(2s_w-1),$$ and if $$\psi_w=\frac{V_{F_w}-1}{V_{F_w}+1},\quad \Phi_w=\frac2{V_{F_w}+1} F_w ,$$
we obtain $F=\frac1{1-\psi_w}P_\HE\Phi_w$.
Moreover, it  is easy to verify that if $w,w'\in \Omega$ with $s_w\ne s_{w'}$, the corresponding factorizations  are different.

\
In order to complete the proof of  Theorem \ref{factorization theorem}
we need to verify the implication (b) $\Rightarrow$ (a) together with  the uniqueness of the factorization in the case when $\|F\|_{\H_k(\HE)}=1$. The argument is based on the following   lemma which contains  a useful  approximation result.

\begin{lemma}
	\label{approx-lemma} Let $\psi\in \Mult(\H_s)$, $\Phi\in \Mult(\H_s,\H_k(\HE))$, and assume that
	\be
	\label{(b)}\|\psi h\|_{\H_s}^2+\|\Phi h\|_{ \H_k(\HE)}^2   \le\|h\|_{\H_s}^2,\quad h\in \H_s.\ee
(i) For $0<r<1$, the function $F^r=\frac{\Phi}{1-r\psi}$ belongs to  $ \Mult(\H_s,\H_k(\HE))$ and satisfies
\be
\label{r-approx}\|F^rh\|_{\H_k(\HE)}^2\le \text{\rm Re }\left\langle\frac{1+r\psi}{1-r\psi}h,h\right\rangle_{\H_s}, \quad h\in \H_s.\ee
(ii) If $|\psi(z_0)|<1$, $F^r$ converges weakly in $\H_k(\HE)$  to  $F=\frac{\Phi}{1-\psi}$  when $r\to 1^-$, and we have  $$s_w(z)\left(\frac{1+\psi(z)}{1-\psi(z)} +
\overline{\frac{1+\psi(w)}{1-\psi(w)}}\right)-2\langle s_wF,s_zF\rangle_{\H_k(\HE)}>>0,\quad z,w\in \Omega.$$

\end{lemma}

\begin{proof}

(i)  Fix $0<r<1$.  From \eqref{(b)}
it follows that  $\psi\in \Mult(\H_s)$ is contractive, i.e. $\frac1{1-r\psi}\in \Mult(\H_s)$.
Since  $\Phi\in \Mult(\H_s,\H_k(\HE))$ and $F^r=\frac{\Phi}{1-r\psi}$ we obtain  that  $ F^r\in \Mult(\H_s,\H_k(\HE))$. Moreover,  \eqref{(b)} also  implies that
$$r^2\|\psi h\|_{\H_s}^2+\|\Phi h\|_{ \H_k(\HE)}^2   \le\|h\|_{\H_s}^2,\quad h\in \H_s,$$
and when applied to $(1-r\psi)^{-1}h$ it yields
$$\|F^rh\|_{\H_k(\HE)}^2\le \left\|\frac1{1-r\psi}h\right\|^2-\left\|\frac{r\psi}{1-r\psi}h\right\|^2.$$
Then the result follows from  $\frac{1+r\psi}{1-r\psi}-1=\frac{2r\psi}{1-r\psi},~\frac{1+r\psi}{1-r\psi}+1
=\frac{2}{1-r\psi}$.\\ (ii)  We have that  $F^r(z)\to F(z)$  when $r\to 1^-$, and an application of  (i) with $h=1=s_{z_0}$,  gives  $\|F^r\|_{\H_k(\HE)}\le \frac{1+|\psi(z_0)|}{1-|\psi(z_0)|}$ which shows that $F^r\to F$ weakly in $\H_k(\HE)$ when $r\to 1^-$. To see the second assertion, let $h\in \H_s$  with $$h=\sum_{i=1}^na_is_{z_i},$$ and apply again (i) to obtain
\begin{align*}& \sum_{i,j=1}^na_i\overline{a}_j\langle s_{z_i}F, s_{z_j}F\rangle_{\H_k(\HE)}=\|hF\|_{\H_k(\HE)}^2 \le
\liminf_{r\to 1^-}\|hF^r\|_{\H_k(\HE)}^2\\& \le \limsup_{r\to 1^-}  \text{\rm Re }\left\langle\frac{1+r\psi}{1-r\psi}h,h\right\rangle_{\H_s} \\
&= \sum_{i,j=1}^na_i\overline{a}_j \frac{1}{2} s_{z_i}(z_j) \left(\frac{1+\psi(z_j)}{1-\psi(z_j)} +\overline{\frac{1+\psi(z_i)}{1-\psi(z_i)}}\right), \end{align*}
which completes the proof.
	\end{proof}
An immediate consequence  is the implication (b) $\Rightarrow$ (a) in Theorem \ref{factorization theorem} (i). Indeed, if (b) holds, then \eqref{(b)} holds and in addition, $\psi(z_0)=0$, hence \eqref{r-approx} with $h=1$ gives  $\|F^r\|_{\H_k(\HE)}^2\le 1$. By part (ii) of the lemma we obtain  $\|F\|_{\H_k(\HE)}^2\le 1$. \\ The uniqueness assertion
in part (ii) of Theorem \ref{factorization theorem}   is another direct application.
 Let $$L(z,w)
=s_w(z)\left(\frac{1+\psi(z)}{1-\psi(z)} +
\overline{\frac{1+\psi(w)}{1-\psi(w)}}\right)-2\langle s_wF,s_zF\rangle_{\H_k(\HE)} $$
denote the positive definite function from part (ii) of Lemma \ref{approx-lemma}.   If $\psi(z_0)=0$ and $\|F\|_{\H_k(\HE)}=1$, it follows that $L(z_0,z_0)=0$. Then the standard inequality $$|L(z,z_0)|^2\le L(z_0,z_0)L(z,z),$$
gives $L(z,z_0)=0,$ hence $$\frac{1+\psi(z)}{1-\psi(z)}=
2\langle F,s_zF\rangle_{\H_k(\HE)}-1=V_F(z),$$
i.e., $\psi=\frac{V_F-1}{V_F+1}$, and the assertion follows.\\ This completes the proof of  Theorem \ref{factorization theorem}.

Finally, we record the following sharpening of the first part of Lemma \ref{approx-lemma} (ii).
Let $F \in \H_k(\HE)$ with $||F||_{\H_k(\HE)} = 1$ and
let $F = \frac{\Phi}{1 - \psi}$ be the unique factorization of Theorem \ref{factorization theorem}.
For $0 < r < 1$, define as above $F^r = \frac{\Phi}{1 - r \psi}$.
Then $F^r$ converges in norm to $F$ as $r \to 1^-$. Indeed, we already saw that
$F^r$ converges weakly to $F$. Moreover, as remarked after the proof of Lemma \ref{approx-lemma}, $||F^r||_{\H_k(\HE)} \le 1$ for all $r < 1$. Since $||F||_{\H_k(\HE)} = 1$,
it follows that the convergence is actually in norm in this case.

\section{Examples and  applications}

\subsection{Examples}
1)  {\it Weighted Bergman spaces.} Let $\mu$ be a finite positive measure on $\B_d,~d\in \N$, such that for all $f\in \Hol(\B_d)$ and for any compact subset $K\subset  \B_d$, there exists $c_K>0$ such that  $$|f(z)|^2\le c_K
\int_{\B_d}|f|^2d\mu,$$
for all $f\in \Hol(\B_d)$ and all $z\in K$.

The corresponding Bergman space   $L_a^2(\mu)=L^2(\mu)\cap \Hol(\B_d)$  is a  Hilbert space with  reproducing kernel $k^\mu$.
 Here we  shall only consider the scalar case, but all considerations extend to the vector-valued version $L_a^2(\mu,\HE)$  defined correspondingly. \\
 Since every analytic contractive $\HB(l^2,\C)$-valued function induces a contractive multiplier on $L_a^2(\mu)$, it follows by Corollary \ref{exCNPfactor} that $k^\mu/s>>0$ for any analytic normalized CNP kernel in $\B_d$. In what follows we shall focus on the Drury-Arveson kernel
  $s=\frac1{1-\langle z,w\rangle}$.
	In this case Theorem \ref{factorization theorem} shows that  any $f\in L^2_a(\mu)$ of norm at most $1$,  can be written as $f=\frac{\varphi}{1-\psi}$, where $\psi \in \Mult(H_d^2)$ with $\psi(0)=0$, with multiplier norm at most 1, and $\varphi\in \Hol(\B_d)$ satisfies that
	$$\int_{\B_d}|g\varphi|^2d\mu\le \|g\|_{H_d^2}^2,$$ i.e. $|\varphi|^2d\mu$ is a \emph{Carleson measure} for $H_d^2$. These objects are best understood  when $d=1$.

 Theorem \ref{factorization theorem} and Corollary \ref{gen Harmonic Majorant} lead to  interesting inequalities:\\
 If $f\in L^2_a(\mu)$ with $\|f\|=1$, then the Sarason function of $f$ is given by  $$V_f(z) = \int_{\B_d} |f(w)|^2 \frac{1+\langle z,w\rangle}{1-\langle z,w\rangle}d\mu$$ and we conclude
$$\|\frac{V_f-1}{V_f+1} h\|^2_{H_d^2} + \|\frac{2fh}{V_f+1}\|^2_{L^2_a(\mu)} \le \| h\|^2_{H_d^2}$$ for all $h \in H_d^2$.  Moreover, Corollary \ref{gen Harmonic Majorant}
 says that $$|f(z)|^2\le \frac{k_z^\mu(z)}{ (1-|z|^2)}\int_{\B_d}|f(w)|^2 \frac{1-|\langle z,w\rangle|^2}{|1-\langle z,w\rangle|^2} d\mu(w).$$

Let us apply Corollary \ref{phi=GGam-appl} when  $\mu=A$, the normalized  area measure on the unit disc $\B_1=\D$.  In this case $L_a^2(A)$ is denoted simply by $L_a^2$ and its reproducing kernel is
$$k^A_w(z)=\frac{1}{(1-\overline{w}z)^2}=s^2_w(z),$$ where
$s_w(z)=\frac1{1-z\overline{w}}$ is the Szeg\"o kernel.
Here we have that   $k^A/s=G(z)G(w)^*$   with $G(z)=(1,z,z^2,...)$. \\Then by  Corollary \ref{phi=GGam-appl}  and the fact that $\Mult(\H_s)=H^\infty$,   it follows that $f\in L^2_a$ with $ \|f\|_{L^2_a}\le 1$,  if and only if there are $\psi, \varphi_n \in H^\infty,~n \ge 1$ with $\psi(0)=0$, $|\psi(z)|^2 +\sum_{n=1}^\infty |\varphi_n(z)|^2 \le 1$,
such that  $$f(z)= \frac{\sum_{n=1}^\infty z^n \varphi_n(z)}{1-\psi(z)}.$$ Similar calculations can be carried out for standard weighted Bergman  on the unit ball $\B_d$, or the polydisc $\D^d$.

2) {\it  Hardy spaces.}  The  Hardy space $H^2(\B_d)$ is defined as the closure of analytic polynomials in $L^2(\sigma_d)$, where $\sigma_d$ is the normalized Lebesgue measure on the unit sphere. Its reproducing  kernel $k$ is given by $k=s^d$, where $s$ is the Drury-Arveson kernel.  All considerations above apply, the corresponding calculations go through and one obtains similar results. For example,
for  $d=2$ Corollary \ref{phi=GGam-appl} implies that
$f$ is in the unit ball of  $H^2(\B_2)$
if and only if
\be
\label{eq2}
f(z) \ = \
\frac{\sum_{\alpha} c_\alpha z^\alpha \varphi_\alpha(z)}{1- \psi(z)},
\ee
where $\alpha = (\alpha_1, \alpha_2)$ is a multi-index,  $\| c_\alpha z^\alpha \|_{H^2_2} = 1$,
$\psi(0) = 0$ and
\[
\begin{pmatrix}
\psi \\
\varphi_{(1,0)} \\
\varphi_{(0,1)} \\
\vdots
\end{pmatrix}
\]
is a contraction from $H^2_2$ to $H^2_2 ( l^2)$. \\
If  $f\in H^2(\B_d),~d\in\N$ has unit norm, its Sarason function is
$$V_f(z) = \int_{|w|=1} |f(w)|^2 \frac{1+\langle z, w\rangle }{1-\langle z, w\rangle}d\sigma_d(w),$$
and Theorem \ref{factorization theorem} gives
$$\|\frac{V_f-1}{V_f+1} h\|^2_{H_d^2} + \|\frac{2fh}{V_f+1}\|^2_{H^2(\B_d)} \le \| h\|^2_{H_d^2}$$ for all $h \in H_d^2$. Also, by Corollary \ref{gen Harmonic Majorant} we  have
$$(1-|z|^2)^{d-1}|f(z)|^2 \le \int_{|w|=1} |f(w)|^2 \frac{1-|\langle z,w\rangle |^2}{|1-\langle z, w\rangle |^2} d \sigma_d(w).$$
In product domains, for example $\D^d,~d\in \N\cup\{\infty\}$,    the  Hardy space $H^2(\D^d)$ is defined as the closure of analytic polynomials in $L^2(\sigma_1^d)$, where $\sigma_1^d$ is the  product of $d$ copies of $\sigma_1$.   In the case when $d=\infty$, we consider analytic polynomials in a finite number of variables. It is well known  (see for example \cite{HLS}), that $H^2(\D^\infty)$ can be identified with the space of Dirichlet series with square summable coefficients. It can be viewed as a reproducing kernel Hilbert space on the set $\Omega$ consisting of points in $\D^\infty$ whose coordinates form an $l^2$-sequence. In all cases the reproducing kernel is given by $$k_w(z)=\prod_{j=1}^ds^0_{w_j}(z_j),\quad z=(z_j),  w=(w_j)$$
where $s^0$ is the Szeg\"o kernel. Clearly,  each factor of this product  is a normalized CNP factor of $k$.
Similar  calculations can be performed using these kernels.
Also, according to   Corollary \ref{exCNPfactor} (i),  if $u\in H^\infty(\D^d)$ of norm at most $1$ with $u(0)=0$, 
 and $s_w(z)=\frac1{1-u(z)\overline{u(w)}}$, then $k/s>>0$.

\subsection{Invariant subspaces}
A direct application of Theorem \ref{factorization theorem} gives (see \cite{AHMcR} for the case $k=s$) that if $k=sg$ with $s$ a normalized CNP kernel  and $g$ positive definite,  then the zero-sets of $\H_k$-functions coincide with the zero-sets of functions in $\Mult(\H_s,\H_k)$.  This is a somewhat surprising result since the second space of functions might be considerably smaller. The idea extends to multiplier-invariant subspaces in a natural way. If $\HE$ is a separable Hilbert space, a closed subspace $\HM$ of $\H_k(\HE)$ is called multiplier-invariant  if $\varphi\HM\subset M$ whenever $\varphi\in \Mult(\H_k)$. The multiplier-invariant subspace generated by $ \mathcal{S}\subset \H_k(\HE)$  is the closure of  in $\H_k$ of  $\{\varphi F:~\varphi\in \Mult(\H_k),~F\in \mathcal{S}\}$
 and will be denoted by $[S]$. We also write $[\{F\}]=[ F]$.
\begin{corollary}\label{inv-subspaces}  If  $F\in \H_k(\HE), \|F\|_{\H_k(\HE)}=1$, and let  $F=\frac1{1-\psi}\Phi$ be the factorization given by Theorem \ref{factorization theorem}. Then  $[ F]=[\Phi]$. In particular,  every multiplier-invariant subspace of $\H_k(\HE)$ is generated by elements of $\Mult(\H_s,\H_k(\HE))$.    \end{corollary}
\begin{proof} Let  $F\in \H_k(\HE), \|F\|_{\H_k(\HE)}=1$ with  $F=\frac1{1-\psi}\Phi$ as in Theorem \ref{factorization theorem}.
Then by definition and  Lemma  \ref{general-criterion}, $\Phi=(1-\psi)F\in [ F]$. Conversely, the same argument shows that  for $0<r<1$,  we have  $F^r=\frac{1}{1-r\psi}\Phi\in [ \Phi]$, hence by  Lemma \ref{approx-lemma} (ii) it follows that $F\in [\Phi]$. The second part of the statement is an obvious consequence of the first.
\end{proof}

\subsection{Extremal functions} Let $k,s$ be reproducing kernels on $\Omega$ such that $s$ is a normalized CNP kernel, $s_{z_0}=1$, and $k/s>>0$.   We say that $F\in \H_k(\HE)$ is extremal if $$\langle \varphi F,F\rangle_{\H_k(\HE)} =\varphi(z_
0)$$
 for all $\varphi\in \Mult(\H_k)$. These functions  generate wandering subspaces for shift-invariant subspaces and  in certain cases they play  an essential role in that theory (see for example,  \cite{ARS}, \cite{Shimorin1}, \cite{Eschetal}).
If $F$ is extremal in $\H_k$ then it has unit norm and  $$V_F(z)=2\langle  F, s_z F\rangle_{\H_k(\HE)}-1=1,$$
hence in the notation in Theorem \ref{factorization theorem} we have $\psi=0$, and $\Phi=F$.

\begin{corollary}\label{extremal}  Let $k,s$ be as above. \\ (i) Every extremal function in $F\in \H_k(\HE)$  is a contractive multiplier  from $\H_s$ into $\H_k(\HE)$. In particular, if $F$  is extremal  in $\H_k(\HE)$, then
	$$\|F(z)\|_{\HE}^2\le \frac{k_z(z)}{s_z(z)}.$$
(ii) If  the linear span of the kernels $s_{z},~z\in \Omega$,  is dense in $\H_k$, then a function $F \in \mathcal H_k(\HE)$ of norm $1$ is an extremal function in $\mathcal H_k(\HE)$ if and only if it is a contractive multiplier from $\mathcal H_s$ into $\mathcal H_k(\HE)$.
	\end{corollary}
\begin{proof} The first part is from   Theorem \ref{factorization theorem},  while the second can be proved either directly, or by an application of  Corollary \ref{gen Harmonic Majorant}. (ii) Suppose that $F$ is a contractive multiplier from $\mathcal H_s$ into $\mathcal H_k(\HE)$.
  By Theorem \ref{factorization theorem} (ii) there is a unique representation $F = \Phi / (1- \psi)$, where $|| \psi h||_{\mathcal H_s}^2 + ||\Phi h||^2_{\mathcal H_k(\HE)} \le ||h||_{\mathcal H_s}^2$ for all $h \in \mathcal H_s$. Since $F$ is a contractive multiplier from $\mathcal H_s$ into $\mathcal H_k(\HE)$, it follows from uniqueness that $\Phi = F$ and $\psi = 0$. In particular, the Sarason function of $F$ satisfies $V_F = 1$, and hence $\langle F, s_z F \rangle_{\mathcal H_k(\HE)} = 1$ for all $z \in \Omega$. The density assumption now implies that $F$ is an extremal function in $\mathcal H_k(\HE)$.
			\end{proof}

Part (ii) extends to the general context a very recent result obtained by D. Seco  \cite{Seco} for the Dirichlet space.  Part (i)  of the corollary recovers some known results, especially when $\HE=\C$.  For example,
 when $k=s$  part (i) can be found in  \cite{McT}. For weighted Bergman spaces on the unit disc  a slightly stronger result holds, since (see subsection  4.1 above) the condition
$$\text{Re }V_F(z)=\int_\D\frac{1-|wz|^2}{|1-\overline{w}z|^2}\|F(z)\|_{\HE}^2d\mu(z)=1,$$
is in general more restrictive than the conclusion of the corollary. Recently,  Corollary \ref{extremal} (i)  has been established in   \cite{Eschetal} in the case when  $\Omega= \mathbb{B}_d,~z_0=0$, multiplication by the coordinates form a $d$-contraction on $\H_k$, and $s$ is the Drury-Arveson kernel. Recall that $k/s>>0$ by Corollary \ref{exCNPfactor} (ii).  In this case extremal functions $F$ satisfy  $\|F(z)\|_\HE^2\le(1-|z|^2)k_z(z)$. By Corollary \ref{extremal}, these results continue to hold when $d=\infty$.  Another situation when the corollary applies is when $\Omega\subset \D^d,~d\in \mathbb{N}\cup\{\infty\}$, and multiplication by each coordinate is a contraction on $\H_k$. If $z_0=0$, and  $s_{jw}(z)=\frac1{1-z_j\overline{w_j}}$ then $k/s_j>>0,~ j\ge 1$ and clearly, $$\H_{s_j}=H^{2,j} =\{h:~h(z)=v(z_j),~v\in H^2\}.$$
Corollary \ref{extremal} implies that extremal functions $F\in \H_k(\HE)$ are contractive multipliers from each $H^{2,j}$ into $\H_k$ and satisfy $\|F(z)\|_\HE^2\le(1-\max_j|z_j|^2)k_z(z)$.\\
The pointwise estimates  for extremal functions can also  be obtained  from
Corollary \ref{exCNPfactor} (i).
\begin{corollary}\label{pointwise_est-extr} Let $k$ be a reproducing kernel on $\Omega$ and set $$\alpha_k(z)=\sup\{ u(z)u^*(z):~ u\in \Mult(\H_k(l^2),\H_k), u(z_0)=0, \|M_u\|\le 1\}.$$
If $F$  is extremal  in $\H_k(\HE)$, then
$$\|F(z)\|_{\HE}^2\le (1-\alpha_k(z))k_z(z).$$
\end{corollary}
\begin{proof} If $\Mult(\H_k(l^2),\H_k)$ contains only constant functions  there is nothing to prove. If there exists a non-zero $u\in \Mult(\H_k(l^2),\H_k)$ with $u(z_0)=0,~\|M_u\|\le 1$, then $s_w(z)=\frac1{1-u(z)u^*(w)}$ is a normalized CNP kernel and  Corollary \ref{exCNPfactor} (i) shows that $k/s>>0$. Then by Corollary \ref{extremal} every extremal function $F\in \H_k(\HE)$ satisfies
	$$\|F(z)\|_{\HE}^2\le (1-u(z)u^*(z))k_z(z),$$
	and the result follows.
	 \end{proof}
In some cases, the function $\alpha_k(z)$ can be easily  estimated. For example, if $\Omega=\B_d, ~z_0=0$, and $\H_k$ consists of analytic functions, it follows that  $\Mult(\H_k(l^2),\H_k)$ is contractively contained in the space of bounded analytic $\HB(l^2,\C)$-valued functions. Let
$u\in \Mult(\H_k(l^2),\H_k),~u(0)=0$, with supremum norm at most $1$, and let $z\in \B_d$ be fixed. Apply the maximum principle to the subharmonic function $$\lambda\mapsto \frac{uu^*(\lambda z)}{|\lambda|^2},\quad |\lambda z|<1,$$
to obtain that   $u(z)u^*(z)\le |z|^2$, i.e.  $\alpha_k(z)\le |z|^2$. Similarly,  if $\Omega=\D^d, ~z_0=0$, and $\H_k$ consists of analytic functions, it follows with the above argument that $\alpha_k(z)\le \max_j |z_j|$, where $z=(z_j)$. These estimates  continue to hold in  the case when $d=\infty$, and become equalities when the identity function belongs to the unit ball of $ \Mult(\H_k(l^2),\H_k)$ (i.e.\ when the multiplication operators by the coordinate functions form a row contraction on $\H_k$), respectively when multiplication by each coordinate is contractive.

\subsection{Multipliers} Let $s$ be a normalized CNP kernel on the nonvoid set $\Omega$, and let $k$ be a reproducing kernel on $\Omega$ with $k/s>>0$.
We are interested in  the space $\Mult(\H_s, \H_k(\HE))$. In most cases  we lack a complete characterization of such multipliers, and our aim is to discuss some sufficient conditions for a function to belong to this space.  Our conditions  are expressed in terms of the Sarason functions $V_F,~F\in \H_k(\HE)$.

\
Proposition \ref{a-factorization} turns out to be useful in this context. A direct application shows  that for  $F\in \H_k(\HE)$,
   $x>0$,  we have           $$\Phi_y=\frac1{V_F+x-iy} F \in \Mult(\H_s, \H_k(\HE)),$$
for all $y\in \R$,  and  the functions   are uniformly bounded in this  multiplier space.  Using this observation we can construct other functions  in $\Mult(\H_s, \H_k(\HE))$ in the following way.  Given  a finite Borel measure  $\mu$ supported on the imaginary axis, we let  $\hat{\mu}$ be  its Cauchy transform,
$$\hat{\mu}(z)=\int_{i\R}\frac{d\mu(iy)}{iy-z},\quad z\in \C\setminus i\R.$$
For $x>0$,  let   $\hat{\mu}_x(z)=\hat{\mu}(x+z)$  be defined in the right half-plane $\{\text{Re }z>0\}$. From above we obtain that   \be
\label{C-transf}\hat{\mu}_x(V_F)F\in  \Mult(\H_s, \H_k(\HE)),\ee
  for any
$F\in \H_k(\HE)$, any $x>0$ and any finite Borel  measure $\mu$ on the imaginary axis.

\
For our next application we need to recall the following notion. We say that the \emph{one-function corona theorem} holds for $\Mult(\H_s)$ if  $\varphi^{-1} \in \Mult(\H_s)$ whenever $\varphi \in \Mult(\H_s)$ and $\varphi$ is bounded below on $\Omega$. The condition is certainly fulfilled if the full corona theorem holds for $\Mult(\H_s)$, in the sense that evaluations at points of $\Omega$ are dense in the maximal ideal space of this algebra. For example, from results in \cite{CSW}, \cite{FangXia} and
 \cite[Theorem 5.4]{RiSunkes}  the one-function corona theorem holds for $\Mult(\H_s)$ if $\Omega=\B_d,~z_0=0$, and   $s_w(z)= (1-\langle z,w\rangle)^{-\gamma},~0< \gamma \le 1$, or $s_w(z)= \frac{1}{\langle z, w\rangle }\log \frac{1}{1-\langle z,w\rangle}$.
On the other hand,   \cite[Theorem 1.5]{AHMcR}   shows that the one-function corona theorem may fail even for (radial) normalized CNP kernels on the unit disc.

\begin{corollary}\label{suffmult}  Assume that  the one-function corona theorem holds for $\Mult(\H_s)$.
	If $F \in \H_k(\HE)$  and  its Sarason function  $V_F$  is  bounded in $\Omega$,
	then $F\in \Mult(\H_s,\H_k(\HE))$, and $V_F\in \Mult(\H_s)$.
\end{corollary}

\begin{proof} Apply Proposition \ref{a-factorization}  with $a=1$, or Theorem \ref{factorization theorem}  for $\|F\|_{ \H_k(\HE)}=1$, to conclude that $\psi=\frac{V_F-1}{V_F+1}\in \Mult(\H_s)$, hence $\frac2{V_F+1}=1-\psi\in \Mult(\H_s)$.  If $V_F$ is bounded, $\frac2{V_F+1}$ is bounded below in $\Omega$, hence by assumption it is invertible in $\Mult(\H_s)$, i.e., $V_F\in \Mult(\H_s)$.
Moreover, since $\frac1{V_F+1}F\in \Mult(\H_s,\H_k(\HE))$ and $V_F\in \Mult(\H_s)$, we obtain
$F\in \Mult(\H_s,\H_k(\HE))$.
\end{proof}

We remark that without the assumption that the one-function corona theorem holds,
the preceding corollary may fail. Indeed, by the results of \cite[Section 5]{AHMcR},
there exists a space $\H_s$ of continuous functions on $\overline{\mathbb D}$  with
a normalized CNP kernel $s$ such that $\Mult(\H_s) \subsetneq \H_s$ (called a Salas space there).
For every
$f \in \H_s \setminus \Mult(\H_s)$, the Sarason function $V_f$ is bounded since the multiplier
norm of $s_z$ is uniformly bounded over $z \in \overline{\mathbb D}$, even
though $f \notin \Mult(\H_s)$.

In several cases  the   condition that the Sarason function  is bounded in $\Omega$,   can be replaced by the weaker assumption that its real part  is bounded.  This   is certainly not sufficient to make $V_F$  a multiplier of $\H_s$, but it sometimes implies that  $F\in \Mult(\H_s,\H_k(\HE))$.

\
The simplest example of this type is when $s$ is the Szeg\"o kernel on $\D$, and $k=s,~\HE=\C$, hence $\H_s=\H_k=H^2$.  If
 $g$ is an unbounded analytic function in the unit disc with $0<a\le \text{ Re }g(z)\le b <\infty$, then there is a bounded outer function $f$ with $|f|^2=\text{ Re }g$ a.e. on the unit circle. Therefore, $f$ is a multiplier of $H^2$, but $V_f$, which agrees with $g$ up to an additive
 constant, is not.
\\
Another example is provided by a recent result in \cite{MaRa} which asserts that if $s$ is the unweighted Dirichlet kernel,  $k$ is the reproducing kernel in some weighted Bergman space $L_a^2(\mu)$ on $\D$,  and  $\text{Re }V_1$ is bounded in $\D$, then  $1\in \Mult(\H_s,\H_k)$, i.e. $\H_s$ is continuously contained in $\H_k$. In fact the argument used in  \cite{MaRa} is based on the very general Lemma 24 in \cite{ArSaRo} and can be extended to arbitrary normalized CNP kernels. For kernels on $\B_d$ of the form   $s_w(z)= \frac{1}{\langle z, w\rangle }\log \frac{1}{1-\langle z,w\rangle} $, or  $s_w(z)= (1-\langle z,w\rangle)^{-\gamma},~0< \gamma < 1$,  or more generally, if $s=s_1^\gamma,~0<\gamma<1$, for some analytic normalized CNP kernel $s_1$, and $k$ a weighted Bergman kernel,  then  $k/s>>0$ by Corollary \ref{exCNPfactor} (iii), and  the boundedness of $\text{Re }V_1$ obviously implies that $V_1$ is bounded, as $\Re s$ and $s$ are comparable for such kernels.
On the other hand,  this is no longer the case  when $s$ is the  Drury-Arveson kernel.

\
We are going to prove that  for a significant class of kernels $k$ and all normalized CNP kernels $s$ with $k/s>>0$, the condition  $\text{\rm Re }V_F$ bounded in $\Omega$, implies  that $F\in \Mult(\H_s,\H_k(\HE))$.\\
Our basic assumption is that the norm on $\H_k$ can be expressed with help of $L^2$-norms of linear differential operators.\\
To be more precise,  assume that $\Omega\subset \R^d$, let $\mu_1,\ldots,\mu_m$ be finite positive  Borel measures on $\Omega$, and for $1\le i\le m$ let  $\HL_i$ be a linear differential operator of the form
$$\HL_i=\sum_{|\alpha|\le N}a_{i,\alpha}\partial^\alpha,$$
where $N\in \N$ is fixed, the coefficients  $a_{i,\alpha}$ are $\mu_i$-measurable functions, and, as usual, $\partial^\alpha=\frac{\partial^{|\alpha|}}{\partial^{\alpha_1}x_1\ldots\partial^{\alpha_d}x_d}$.

\
Now assume that $\H_k$ contains a dense set $\HD$, such that the functions in $\HD$, together with all multipliers in $\Mult(\H_s)$ are continuous on $\Omega$ and have partial derivatives of order $\le N$  $\mu_i$-a.e.,
for $1\le i\le m$. Moreover, assume that there are absolute constants $c_1,c_2>0$ with
\be
\label{diff-op-norm} c_1\|f\|_{\H_k}^2\le  \sum_{i=1}^m\int_\Omega |\HL_if|^2d\mu_i\le c_2\|f\|_{\H_k}^2,\quad f\in \HD.\ee
 Note that in  this case each $\HL_i$ extends to a bounded linear operator from $\H_k$ into $L^2(\mu_i)$. If we denote these extensions  by $\tilde{\HL}_i$, it follows that
\be
\label{diff-op-norm-ext} \|f\|^2=\sum_{i=1}^m\int_\Omega |\tilde{\HL}_if|^2d\mu_i,\ee
defines an equivalent norm on $\H_k$.

Sobolev spaces provide standard examples of  such spaces. Also, most  of the examples considered in this paper satisfy these assumptions; weighted Bergman and Hardy spaces on $\B_d$, or $\D^d$, weighted Dirichlet spaces on the unit disc, or more generally weighted Besov spaces on $\B_d$, and in particular the Drury-Arveson spaces $H_d^2,~d\in \N$  (see \cite{ZZ}).

\begin{theorem}\label{diff-op-thm}
Let $\Omega, \mu_i, \HL_i,~1\le i\le m$,  $\H_k$ and $\HD$  be as above and assume that \eqref{diff-op-norm}  holds.  If  $F\in  \H_k(\HE)$ and $\text{\rm Re }V_F$  is bounded in $\Omega$,  then $F\in \Mult(\H_s,\H_k(\HE))$, and there exists a constant $c_N>0$ depending only on $N$,  such that
$$\|F\|_{\Mult(\H_s,\H_k(\HE))}\le c_N(\|\text{\rm Re }V_F\|_\infty +3)^{N+\frac1{2}}.$$
 \end{theorem}
\begin{proof}
For   $0\le r\le m$,  let $\tilde{\HL}_r$ be  the extended differential operator  from \eqref{diff-op-norm-ext}. We shall prove that there exists a constant $C_N>0$, depending only on $N$,  such that whenever $F\in \H_k(\HE)$ and $h\in \H_s$ is a finite linear combination of reproducing kernels in $\H_s$, we have
\be
\label{key-est-diff}
\int_\Omega \frac{\|(\tilde{\HL}_r\otimes 1_\HE )h F\|_\HE^2}{(\text{Re }V_F+3)^{2N+1}}d\mu_r\le C_N\|h\|_{\H_s}^2. \ee
 Clearly, the theorem follows directly from this inequality.	\\
Note first that it will be sufficient to prove \eqref{key-est-diff} for $F$ in dense subset  of $\H_k(\HE)$, since $h\in \Mult(\H_s)\subset \Mult(\H_k(\HE))$, and  if $F_n\to F$ in $\H_k(\HE)$, then $V_{F_n}(z)\to V_F(z),~z\in \Omega$. Then \eqref{diff-op-norm-ext} together with Fatou's lemma show that the estimate holds for arbitrary $F\in \H_k(\HE)$ with the same constant $C_N$. Consequently, for a fixed orthonormal basis $\{e_n\}$ of $\HE$ we can consider $\HE$-valued functions $F$ of the form $$F= \sum_{n\le M}f_ne_n,$$
with $f_n\in\HD,~1\le n\le M$, the dense subset of $\H_k$ from our assumption. In this case $F$ has $\HE$-valued  partial derivatives of order $\le N$ $\mu_r$-a.e., and these are $\mu_r$-measurable.

\
Start with  the inequality in Proposition \ref{a-factorization} to obtain for $\text{Re }a>0$,
\begin{align*}4\text{\rm Re }a\left\|\frac{1}{V_{F}+\overline{a}}hF\right\|_{ \H_k(\HE)}^2 &\le \|h\|_{\H_s}^2-
\left\|\frac{V_{F}-a}{V_{F}+\overline{a}}h\right\|_{\H_s}^2\\&\le 4 \Re a \Re \left\langle \frac{1}{V_{F}+\overline{a}}h, h \right\rangle_{\H_s},\end{align*}
and from   \eqref{diff-op-norm}
$$\int_\Omega \left\|\HL_r\frac{1}{V_{F}+\overline{a}}hF\right\|_\HE^2d\mu_r\lesssim
\text{Re }\left\langle \frac{1}{V_{F}+\overline{a}}h,h \right\rangle_{\H_s}. $$
Now let $a=x+iy$, with $x>0$   fixed, but arbitrary.  Integration in $y$ yields
\begin{align}\label{crucial step}\int_\Omega\int_{-\infty}^\infty &\left\|\HL_r\frac{1}{V_{F}+x-iy}hF(z)\right\|_\HE^2dy\mu_r(z)\\&\nonumber \lesssim
\int_{-\infty}^{\infty}\text{Re }\left\langle \frac{1}{V_{F}+x-iy}h,h \right\rangle_{\H_s} dy. \end{align}
We claim that
\be
\label{right hand side} \int_{-\infty}^{\infty}\text{Re }\left\langle \frac{1}{V_{F}+x-iy}h,h \right\rangle_{\H_s} dy
=\pi\|h\|_{\H_s}^2.
\ee
 Indeed, if $h=\sum_jc_js_{z_j}$,
\begin{align*}&\int_{-T}^{T}\text{Re }\left\langle \frac{1}{V_{F}+x-iy}h,h \right\rangle_{\H_s} dy\\&=\frac1{2}\sum_{j,l} c_j\overline{c}_ls_{z_j}(z_l)
\int_{-T}^T\left( \frac{1}{V_{F}(z_l)+x-iy} + \frac{1}{\overline{V_{F}}(z_l)+x+iy}\right)dy\\&
=\sum_{j,l} c_j\overline{c}_l s_{z_j}(z_l)\frac1{2i}\left(\log \frac{V_{F}(z_l)+x +iT}{V_{F}(z_l)+x -iT} +
\log \frac{\overline{V_{F}}(z_j)+x +iT}{\overline{V_{F}}(z_j)+x -iT}\right)\\&\to \pi\sum_{j,l} c_j\overline{c}_ls_{z_l}(z_j),\quad T\to \infty,\end{align*}
and the claim follows by the monotone convergence theorem.\\
In order to estimate the left  hand side of \eqref{crucial step} from below, we note first that by assumption, $V_{F}$ has partial derivatives of order $\le N$ $\mu_r$-a.e., as $\frac{2}{V_F+1} \in \Mult(\H_s)$,
hence by the product rule we have $\mu_r$-a.e.
\begin{align*}\HL_r\frac{1}{V_{F}+x-iy}hF
&=\frac1{V_{F}+x-iy}\HL_rhF\\&+\sum_{j=1}^N\frac1{(V_{F}+x-iy)^{j+1}}\HL_{rj}hF,\end{align*}
where $\HL_{rj}$ are linear differential operators of order $\le N$ with $\mu_r$-measurable coefficients.\\
Now recall that  $y\mapsto\frac{\text{Re } V_{F}(z)+x}{\pi|V_{F}(z)+x+it -iy|^2},~z\in \Omega,t\in \R$ is the Poisson kernel for the right half-plane at the point $w=V_{F}(z)+x+it$, and note that for $2|t|\le x$
  $$\frac{\text{Re } V_{F}(z)+x}{\pi|V_{F}(z)+x+it -iy|^2} \le 2\frac{\text{Re } V_{F}(z)+x}{\pi|V_{F}(z)+x -iy|^2}.$$
 Moreover, for $\mu_r$-almost every $z\in \Omega$ the function
$$U_z(w)= \HL_rhF(z)+\sum_{j=1}^N\frac1{(V_{F}+x+\overline{w})^{j}}\HL_{rj}hF(z),\quad \text{Re }w>0,$$
is anti-analytic  in the right  half-plane, continuous and  bounded in its closure. Consequently, $\|U_z\|_\HE^2$ is subharmonic, bounded and continuous in the closure of the right half-plane, hence its Poisson integral is a (the least)  harmonic majorant in the right half-plane. In particular,
$$ \|U_z(V_{F}(z)+x+it)\|_\HE^2\le \int_{-\infty}^\infty\frac{\text{Re } V_{F}(z)+x}{\pi|V_{F}(z)+x+it-iy|^2}\|U_z(iy)\|_\HE^2dy.
$$
From the last two inequalities we obtain for all $x\in [1,2],~t\in [-\frac{1}{2},\frac{1}{2}]$, and  $\mu_r$-a.e. on $\Omega$
\begin{align*}&\int_{-\infty}^\infty \left\|\HL_r\frac{1}{V_{F}+x-iy}hF(z)\right\|_\HE^2dy\\&=\frac{\pi}{\text{Re } V_{F}(z)+x}\int_{-\infty}^\infty\frac{\text{Re } V_{F}(z)+x}{\pi|V_{F}(z)+x-iy|^2}\|U_z(iy)\|_\HE^2dy\\&
\ge \frac{\pi}{2(\text{Re } V_{F}(z)+x)}\int_{-\infty}^\infty\frac{\text{Re } V_{F}(z)+x}{\pi|V_{F}(z)+x+it-iy|^2}\|U_z(iy)\|_\HE^2dy\\&
 \ge
 \frac{\pi}{2(\text{Re } V_{F}(z)+x)}  \|U_z(V_{F}(z)+x+it)\|_\HE^2\\&
\ge \frac{\pi \|\sum_{j=0}^N(2\text{Re }V_{F}+2x-it)^{N-j}\HL_{rj}hF(z)\|_\HE^2}{2^{2N+1}(\text{Re } V_{F}(z)+3)^{2N+1}},\end{align*}
where  $\HL_{r0}=\HL_r$.  Thus, if $d\nu=\frac{d\mu_r}{(\text{Re }V_F+3)^{2N+1}}$,  this  estimate together with \eqref{right hand side} and \eqref{crucial step}  gives
$$ \left\|\sum_{j=0}^N(2\text{Re }V_F+2x-it)^{N-j} \mathcal L_{rj}hF\right\|_{L^2(\nu,\HE)}^2 \le A_N\|h\|_{\H_s}^2,$$
for some constant $A_N>0$ depending only on $N$, and for  all $2x-it \in [2,4]\times[-\frac1{2},\frac1{2}]$. In particular, for every $u\in L^2(\nu,\HE)$, the polynomial
$$p_u(w)=\left\langle\sum_{j=0}^N(2\text{Re }V_F+w)^{N-j}\mathcal L_{rj}hF, u\right\rangle_{L^2(\nu,\HE)},$$
satisfies
$$|p_u(2x-it)|
\le A_N^{1/2}\|u\|_{L^2(\nu,\HE)}\|h\|_{\H_s},\quad 2x-it \in [2,4]\times[-\frac1{2},\frac1{2}].$$
Then $$|p^{(N)}_u(3)|\le B_N\|u\|_{L^2(\nu,\HE)}\|h\|_{\H_s},$$
with $B_N>0$ depending only on $N$, which implies \eqref{key-est-diff} and completes the proof.
\end{proof}

The result is of interest even in the special case when $k=s$ and $\HE=\C$. Actually, the main motivation for this theorem was the Drury-Arveson kernel, and we record the corresponding result as an immediate application.

\begin{corollary}\label{Drury-Arveson} If  $s_w(z)=\frac1{1-\langle z,w\rangle}, ~z,w\in\B_d$, and  $f\in H_d^2$ satisfies $$\sup_{z\in \B_d}\text{\rm Re }\langle f, s_zf\rangle_{H_d^2}<\infty,$$
then $f\in \Mult(H_d^2)$.
\end{corollary}

Recall from Equation \eqref{Harmonic Majorant} that $\frac{||s_z f||^2}{||s_z||^2} \le \Re V_f(z)$ for all $f \in \H_s$. While Corollary \ref{Drury-Arveson} shows that boundedness of $\Re V_f$
implies that $f \in \Mult(H^2_d)$, the main result of \cite{FX15} shows that boundedness of
$\frac{||s_z f||^2}{||s_z||^2}$ is not sufficient for $f$ to belong to $\Mult(H^2_d)$.

The natural question which arises is whether 
  $\text{Re }V_F$ is bounded for all $F\in \Mult(\H_s,\H_k(\HE))$?
As we shall see below, the answer is negative.

Another  direct consequence of   Theorem \ref{diff-op-thm} concerns the embedding of $\H_s$ into $\H_k$.
\begin{corollary}\label{embedding}
Let $\Omega, \mu_i, \HL_i,~1\le i\le m$  and $\H_k$  be as in Theorem \ref{diff-op-thm}. Let $s$ be a normalized CNP kernel such that $k/s>>0$. If  $1\in  \H_k$ and $\text{\rm Re }V_1$  is bounded in $\Omega$,  then $\H_s$ is continuously contained in $\H_k$.
 \end{corollary}
As pointed out in the Introduction, if $\H_k$ is a weighted Bergman space, as defined in subsection 4.1,
this extends the result in \cite{MaRa}.  The general framework for Carleson  embeddings involves $L^2$-spaces rather than  Bergman spaces, but the additional step is trivial  in many cases. For example, if $s$ is an analytic  normalized CNP kernel on $\B_d$ and $\mu$ is an arbitrary finite positive Borel measure on $\B_d$ satisfying
\be
\label{re v1} \sup_{z\in \B_d}\text{Re } \int_{\B_d} s_z d\mu < \infty, \ee
then  Corollary \ref{embedding} easily implies that $\mu$ is a Carleson measure for $\H_s$.   Indeed,  note that if $v$ denotes  the Lebesgue measure on $\B_d$, then $\H_s$ is continuously embedded in $L_a^2(\mu_0)$, where $$d\mu_0(z) =\max_{|w|=|z|}s_w(w)^{-1/2}dv(z).$$
Then $\H_k=L_a^2(\mu_0+\mu)$ is a weighted Bergman space with $1\in \H_k$ and since $\mu_0$ is radial, $$\int_{\B_d}s_zd\mu_0=\mu_0(\B_d),$$ hence by \eqref{re v1} it follows
that $\text{Re }V_1$  bounded in $\B_d$.
Thus $\H_s$ is continuously contained in $\H_k$ and consequently, it is also continuously contained in $L^2(\mu)$. In concrete cases, \eqref{re v1} together with the method in \cite{MaRa} can be used to derive one-box conditions for  Carleson measures for $\H_s$.

\subsection{Sarason functions in $D(\mu)$-spaces}  In the general context considered above it is difficult to compute the Sarason function, or even to estimate its real part. There is an important class of spaces where the second  problem appears more tractable due to the work of Shimorin (\cite{Shimorin},\cite{Shimorin2}).

The local Dirichlet integral of  $f\in H^2$ at $\zeta\in \overline{\D}$ is defined by
\begin{equation*} D_\zeta(f) = \int_{\T} \left| \frac{f(z) - f(\zeta)}{z - \zeta} \right|^2 d m(z).\end{equation*} 
Here $m$ denotes the normalized arclength measure on the unit circle $\T$, and if $\zeta\in \T$, the value $f(\zeta)$ is the nontangential limit of $f$ at $\zeta$ which exists whenever $D_\zeta(f)$ is finite (see \cite{RiSundberg}). Given a finite positive Borel measure on $\overline{\D}$,  $D(\mu)$ is the Hilbert space consisting of all $H^2$-functions  $f$ with $$\|f\|_{D(\mu)}^2=\|f\|_{H^2}^2+\int_{\overline{\D}} D_\zeta(f)d\mu(\zeta)<\infty.$$
These norms can be expressed with help of the first derivative. We have (see \cite{AA}, \cite{Richter})
$$\|f\|_{D(\mu)}^2=\|f\|_{H^2}^2+\int_\D |f'(z)|^2U_\mu(z)dA(z),$$
where $$U_\mu(z)=\int_\D\log\left|\frac{1-\overline{\zeta}z}{z-\zeta}\right|d\mu(\zeta)+\int_{\T}\frac{1-|z|^2}{|1-\overline{\zeta}z|^2}d\mu(\zeta).$$
$D(\mu)$-spaces appeared first in \cite{Richter} for measures supported on $\T$ in connection with functional models for two-isometries. The general case was considered in \cite{AA}. The most common examples are the standard weighted Dirichlet  spaces $D_\alpha,~0\le \alpha<1$, consisting of analytic  functions in $\D$ with \be
\label{deriv-norm-D}\|f\|_\alpha^2=\|f\|_{H^2}^2+\int_\D|f'(z)|^2(1-|z|^2)^\alpha dA(z)<\infty.\ee
If  $\mu_0=m$, and for $0<\alpha<1$, 
$d\mu_\alpha=-(1-|z|^2)\Delta(1-|z|^2)^\alpha dA$, where $\Delta$ denotes the Laplacian, we   have  $D_\alpha=D(\mu_\alpha)$, with equality of norms.

Shimorin proved in \cite{Shimorin2} that for every finite positive Borel measure on $\overline{\D}$, the reproducing kernel $s^\mu$ in  $D(\mu)$ is a normalized CNP kernel. Moreover, in \cite{Shimorin},  Proposition 3 and Corollary 4,  he showed that for every $f\in D(\mu)$ we have 
  \begin{align}\label{Shimorin_formula}
    \Re V_f(z)& = \int_\T \frac{(1-|z|^2)|f(\zeta)|^2}{|1-\overline{\zeta}z|^2}dm(\zeta)\\&\nonumber + \int_{\overline{\D}} (2 \Re s_z^\mu(\zeta) - 1) D_\zeta(f) \, d\mu (\zeta).
  \end{align}

We shall use this remarkable identity to show that the converse of Theorem \ref{diff-op-thm} fails in $D_\alpha,~0<\alpha<1$.
\begin{prop}\label{unbounded sarason} For $0<\alpha<1$, 
 there exists  $u\in \Mult(D_\alpha)$ such that $\Re V_u$ is unbounded in $\D$.\end{prop}
 The result holds for $\alpha=0$ as well, but the proof is more involved and will be omitted. \\The proof of the  proposition requires some preliminary observations.  Throughout in what follows we shall assume that $0<\alpha<1$, and denote by $s^\alpha$ the reproducing kernel in $D_\alpha$. Note that this is a radial kernel, i.e.  
\be
\label{s_w}s^\alpha_w(z)=1+\sum_{n=1}^\infty c_n^\alpha (\overline{w}z)^n,\quad c_n^\alpha\ge 0,\, c_n^\alpha\sim (n+1)^{\alpha-1}.\ee
Therefore we can consider the analytic function $$s^\alpha_1(z)=1+\sum_{n=1}^\infty c_n^\alpha z^n,\quad z\in\D.$$
Some properties of this  function are listed below.

\begin{lemma}\label{properties of s_1} (i) We have  $\Re s_1^\alpha (z)> \frac1{2}$ for all $z\in \D$, and $1-\frac1{s_1^\alpha}$ is a contractive multiplier of $D_\alpha$.\\
	 (ii) $ s_1^\alpha$ and $(s_1^\alpha)'$ are positive on  $[0,1)$ and   
	$$s_1^\alpha(r)\sim (1-r)^{-\alpha},\quad (s_1^\alpha)'(r)\sim (1-r)^{-\alpha-1},$$
where the constants involved depend only on $\alpha$.\\
	(iii) There exist $\varepsilon_\alpha \in (0,1)$ and $\delta_\alpha>0$ such that 
	$$\Re s_1^\alpha(z) \ge \delta_\alpha (1-|z|)^{-\alpha},$$
	whenever $z$ belongs to the set  $S=\{z\in \D:~|z-|z||<\varepsilon_\alpha(1-|z|)\}$.
	\end{lemma}

\begin{proof} 
(i) Since  $s^\alpha$ is a normalized CNP kernel,  by definition and Lemma \ref{CNP} (i) it follows that
both statements are true with $s_w^\alpha$ in place of $s_1^\alpha$. Then the assertion follows by letting $w\to 1$ and using Fatou's lemma in \eqref{deriv-norm-D}.

(ii) is a straightforward application of \eqref{s_w} combined with the standard fact that for $s < 1$,
\begin{equation*}
  \sum_{n=1}^\infty n^{-s} r^n \sim (1-r)^{s-1}
\end{equation*}
as $r \to 1$.

(iii) The same  straightforward estimate in \eqref{s_w} gives $$ (s_1^\alpha)'(z)\le M_\alpha (1-|z|)^{-\alpha-1}, \quad z\in \D$$
for some constant $M_\alpha > 0$.
Now if    $0<\varepsilon<1$ and $|z-|z||<\varepsilon(1-|z|)$, we have
\begin{align*}
|\Re s_1^\alpha(z)- \Re s^\alpha_1(|z|)|& \le |z-|z|| \sup_{t\in [0,1]} |(s_1^\alpha)'(t|z|+(1-t)z)|  \\&\le M_\alpha
\varepsilon(1-|z|)(1-|z|)^{-\alpha-1}\\&= M_\alpha \varepsilon(1-|z|)^{-\alpha}.
\end{align*}
Using (ii),  we choose $\varepsilon_\alpha\in (0,1)$ such that $$\varepsilon_\alpha M_\alpha (1-|z|)^{-\alpha}
\le \frac1{2}\Re s_1^\alpha(|z|)$$ for all $z \in \D$,
and the result follows. 
\end{proof}

The key step for our construction is the following estimate derived from Shimorin's identity.

\begin{lemma}\label{estimate-below}  Let $S$  be the set in Lemma \ref{properties of s_1} (iii).  There exists $c_\alpha>0$ such that for all $f\in D_\alpha$,
$$\|f\|_\alpha^2+\sup_{z\in\D}\Re V_f(z) \ge c_\alpha \int_{S} |f'(\zeta)|^2 dA(\zeta).$$	
\end{lemma}
\begin{proof}
Note that \eqref{Shimorin_formula} implies for all $f\in D_\alpha$,
$$ \|f \|_\alpha^2+  \Re V_f(z)\ge \int_{\D} 2\Re s_z^\alpha(\zeta)  D_\zeta(f) \, d\mu_\alpha (\zeta),$$
  hence by Fatou's lemma 
\begin{align*} \|f \|_\alpha^2+  \sup_{z\in\D }\Re V_f(z)&\ge \liminf_{z\to 1}\int_{\D} 2\Re s_z^\alpha(\zeta)  D_\zeta(f) \, d\mu_\alpha (\zeta)\\&
\ge \int_{\D} 2\Re s_1^\alpha(\zeta)  D_\zeta(f) \, d\mu_\alpha (\zeta).\end{align*}
The standard estimate $(1-|\zeta|^2)|h(\zeta)|^2\le \|h\|_{H^2}^2$,
yields $$D_\zeta(f)\ge (1-|\zeta|^2)|f'(\zeta)|^2.$$
Thus
$$ \|f \|_\alpha^2+  \sup_{z\in\D }\Re V_f(z)\ge \int_{\D} 2\Re s_1^\alpha(\zeta) |f'(\zeta)|^2(1-|\zeta|^2) d\mu_\alpha (\zeta),$$
and since 
$d\mu_\alpha=-(1-|z|^2)\Delta(1-|z|^2)^\alpha dA$, the result follows by Lemma \ref{properties of s_1}  (iii).
\end{proof}

\begin{Pf} {\it Propostion \ref{unbounded sarason}.} Let $f=1-\frac1{s_1^\alpha}$ and recall from Lemma \ref{properties of s_1} (i) that $f\in \Mult(D_\alpha)$ with $\|M_f\|\le 1$.  Since $|f(z)|<1,~z\in \D$, it follows that $(M_f^*)^ns^\alpha_z=(\overline{f(z)})^ns_z^\alpha$ converges to zero in $D_\alpha$, hence  $(M_f^*)^n$ converges to zero in the strong operator topology. Thus $M_f$ admits a contractive weak-$*$-weak-$*$ continuous
  $H^\infty$-functional calculus. Since this functional calculus extends the polynomial functional calculus,
  it is easy to check that $g(M_f) = M_{g \circ f}$ for all $g \in H^\infty$, i.e. $g\circ f\in \Mult(D_\alpha)$ for all $g\in H^\infty$ with $\|g\circ f\|_{\Mult(D_\alpha)}\le\|g\|_\infty$ (a similar construction
  appears in \cite[Lemma 12]{DE12}).
We claim that if $B$ is an infinite interpolating Blaschke product with zeros in $[0,1)$ then $\Re V_{B\circ f}$ is unbounded in $\D$.\\
Let $\{z_n:n\ge 1\}$ be the zero-set of $B$ and use Lemma \ref{properties of s_1} (ii) to conclude that $z_n=f(w_n),~n\ge 1$, with $w_n\in [0,1),~\lim_{n\to\infty} w_n= 1$. Then, with the notations in  Lemma \ref{properties of s_1} (iii), there exists an infinite set $J\subset \mathbb{N}$ such that the  discs $\Delta_n=\{|z-w_n|<\frac{\varepsilon_\alpha}{3}(1-w_n)\},~n\in J$, are disjoint.
Note that  by the triangle inequality we have for  $z\in \Delta_n$
$$1-|z|>(1-\frac{\varepsilon_\alpha}{3})(1-w_n),\,\,|z-|z||<2|z-w_n|<\frac{2\varepsilon_\alpha}{3}(1-|w_n|)<\varepsilon_\alpha(1-|z|),$$
so that $\Delta_n\subset S$. 
 Since $\|B\|_\infty=1$,   Lemma \ref{estimate-below} gives
$$1+\sup_{z\in\D}V_{B\circ f}(z)\ge c_\alpha\sum_J\int_{\Delta_n}|(B\circ f)'|^2dA,$$
and since  $|(B\circ f)'|^2$ is subharmonic in $\D$, we obtain
$$1+\sup_{z\in\D}V_{B\circ f}(z)\ge \frac{\pi\varepsilon_\alpha^2 c_\alpha}{9}\sum_J(1-w_n)^2|(B\circ f)'(w_n)|^2.$$ 
It suffices to show that the values  $(1-w_n)|(B \circ f)'(w_n)|,~n\in \mathbb{N}$, are bounded below.\\
  Since $B$ is interpolating, there exists $\delta>0$ such that $$|B'(z_n)|\ge\delta (1-z_n),\quad n\in \mathbb{N}.$$ From $z_n=f(w_n)$ it follows that $1-z_n=\frac1{s_1^\alpha(w_n)}$, hence 
$$(1-w_n)|(B \circ f)'(w_n)|=(1-w_n)|B'(z_n)|\frac{(s_1^\alpha)'(w_n)}{(s_1^\alpha(w_n))^2}\ge \delta(1-w_n)
\frac{(s_1^\alpha)'(w_n)}{s_1^\alpha(w_n)},$$
and the result follows by Lemma \ref{properties of s_1} (ii).
\end{Pf}

\end{document}